\documentclass[12pt,reqno]{amsart}

\title[]{Girsanov theory under a finite entropy condition}
\author{Christian L\'eonard}

\usepackage{amssymb, amsmath, amsfonts, latexsym, enumerate}
\usepackage{stmaryrd} 

\newtheorem{theorem}{Theorem}
\newtheorem{lemma}[theorem]{Lemma}
\newtheorem{proposition}[theorem]{Proposition}
\newtheorem{corollary}[theorem]{Corollary}

\newtheorem{definition}[theorem]{Definition}

\theoremstyle{remark}

\newtheorem{remarks}[theorem]{Remarks}

\numberwithin{theorem}{section}

\newcommand{\RR}{\mathbb{R}}
\newcommand{\1}{\mathbf{1}}

\newcommand\pf{_{\#}}
\newcommand\as{\textrm{-a.s.}}
\renewcommand\ae{\textrm{-a.e.}}

\newcommand\seq[2]{(#1_#2)_{#2\ge1}}

\newcommand\Lim[1]{\lim_{#1\rightarrow\infty}}

\newcommand{\cadlag}{c\`adl\`ag}

\newcommand\Rd{\RR^d}

\newcommand\ii{{[0,1]}}

\newcommand\OO{\Omega}

\newcommand\iO{{\ii\times\OO}}
\newcommand{\iRO}{\ii\times\Rd_*\times\OO}
\newcommand{\iR}{\ii\times\Rd_*}

\newcommand\IO{\int_\OO}
\newcommand\Ii{\int_\ii}

\newcommand{\IiR}{\int_{\iR}}

\newcommand\II[2]{\int_{[#1,#2]}}

\newcommand\XXX[2]{X_{[#1,#2]}}
\newcommand\XXx[2]{X_{[#1,#2)}}
\newcommand\xXX[2]{X_{(#1,#2]}}

\renewcommand\lg[2]{#1_{#2^-}}

\newcommand\dPR{\frac{dP}{dR}}

\newcommand\xtn{X^{\tau_n}}

\newcommand{\LL}{L}
\newcommand{\KK}{K}
\newcommand{\Lb}{\overline{\LL}}
\newcommand{\Kb}{\overline{\KK}}
\newcommand{\lLL}{{\ell \LL}}
\newcommand{\mut}{\widetilde{\mu}^\LL}
\newcommand{\mul}{\mu^X}
\newcommand{\muh}{\widetilde{\mu}^\lLL}
\newcommand{\mutk}{\widetilde{\mu}^\KK}

\newcommand{\MP}{\mathrm{MP}}
\newcommand{\LK}{\mathrm{LK}}
\newcommand{\lp}{{\lambda^+}}
\newcommand{\lm}{{\lambda^-}}

\begin{document}


\address{Modal-X. Universit\'e Paris Ouest. B\^at.\! G, 200 av.
de la R\'epublique. 92001 Nanterre, France}
 \email{christian.leonard@u-paris10.fr}

\begin{abstract}
This paper is about Girsanov's theory. It (almost) doesn't contain new results but it is based on a simplified new approach which takes advantage of the (weak) extra requirement that some relative entropy is finite. Under this assumption, we present and prove all the standard results pertaining to the absolute continuity of two continuous-time processes on $\Rd$ with or without jumps. We have tried to give as much as possible a self-contained presentation. 

The main advantage of the finite entropy strategy is that it allows us to replace martingale representation results by the simpler  Riesz representations of the dual of a Hilbert space (in the continuous case) or of an Orlicz function space (in the jump case). 
\end{abstract}

\maketitle
    \tableofcontents


\section{Introduction}

This paper is about Girsanov's theory. It (almost) doesn't contain  new results but it is based on a simplified new approach which takes advantage of the (weak) extra requirement that some relative entropy is finite. Under this assumption, we present and prove all the standard results pertaining to the absolute continuity of two continuous-time processes on $\Rd$ with or without jumps.

This article intends to look like lecture notes and we have tried to give as much as possible a self-contained presentation of Girsanov's theory. The author hopes that it could be useful for students and also to readers already acquainted with stochastic calculus.

The main advantage of the finite entropy strategy is that it allows us to replace martingale representation results by the simpler  Riesz representations of the dual of a Hilbert space (in the continuous case) or of an Orlicz function space (in the jump case). The gain is especially interesting in the jump case where martingale representation results are not easy, see \cite{Jac75}. Another feature of this simplified approach is that very few about exponential martingales is needed.

Girsanov's theory studies the relation between a
reference process $R$ and another process $P$ which is assumed to be
absolutely continuous with respect to $R.$ In particular, it is
known that if $R$ is the law of an $\Rd$-valued semimartingale,
then
$P$ is also the law of a semimartingale. In its wide meaning, this theory  also provides
us with a formula for the Radon-Nikod\'ym density $\dPR.$ 

In this article,
we assume that the probability measure $P$ has its relative entropy with respect to $R$:
\begin{equation*}
    H(P|R):=\left\{
    \begin{array}{ll}
E_P\log\left(\dPR\right)\in[0,\infty] & \textrm{if }P\ll R \\  
+\infty& \textrm{otherwise,} \\
    \end{array}
    \right.
\end{equation*}
 which is finite, i.e.\
\begin{equation}\label{eq-02}
    H(P|R)= E_R\left[\dPR\log\left(\dPR\right)\right]<\infty.
\end{equation}
In comparison, requiring $P\ll R$ only amounts to assume that
\begin{equation}\label{eq-10}
    E_R\left(\dPR\right)<\infty
\end{equation}
since $P$ has a finite mass. We are going to take advantage of
the only
difference between \eqref{eq-02} and \eqref{eq-10} which is the
stronger integrability property carried by the extra term
$\log\dPR.$

A key argument of this approach is the variational representation of the relative entropy which is stated at Proposition \ref{res-02}. Some versions of this result are well-known and widely used.  We decided to give a (usually unknown) complete picture of this very useful variational representation together with a complete elementary proof. We think that this complete picture is interesting in its own right.

A clear exposition of the general Girsanov's theorems, with no explicit expression of $\dPR$ in terms of the characteristics of the processes, is given in P. Protter's textbook \cite{Pro04}. The most complete results about Girsanov's theory for $\Rd$-valued processes, including explicit formulas for $\dPR,$ are available in  J. Jacod's textbook  \cite{Jac79}. An alternate presentation of this realm of results is also given in the later book by J. Jacod and A. Shiryaev \cite{JaShi87}. A good standard reference in the continuous case is D. Revuz and M. Yor's textbook \cite{RY99} about continuous martingales.

Next Section 2 is devoted to the statement of the main results. At Section 3, we state and prove the  above mentioned variational representation of the relative entropy.  At Sections 4 and 5, we present the proofs of  Theorems \ref{res-01} and \ref{res-04} which correspond to the continuous case. At Section 6, we give the proofs of  Theorems \ref{res-01b} and \ref{res-04b} which correspond to the jump case.

\section{Statement of the results}

We distinguish the cases where the sample paths are continuous
and
where they exhibit jumps.

\subsection*{Continuous processes in $\Rd$}

The paths which we consider are built on the time interval
$\ii$.
An $\Rd$-valued continuous  stochastic process is a random
variable taking its values in the set $$\OO=C(\ii,\Rd)$$ of all
continuous paths from $\ii$ to $\Rd.$ The canonical process
$(X_t)_{t\in\ii}$ is defined by
$$
    X_t(\omega)=\omega_t,\quad
    t\in\ii,\omega=(\omega_s)_{s\in\ii}\in\OO.
$$
In other words, $X=(X_t)_{t\in\ii}$ is the identity on $\OO$ and
$X_t:\OO\to\Rd$ is the
$t$-th projection. The set $\OO$ is
endowed
with the $\sigma$-field $\sigma(X_t;t\in\ii)$ which is generated
by the canonical
projections. We also consider the canonical
filtration $\Big(\sigma(\XXX0t);t\in\ii\Big)$ where for each
$t,$
$\XXX0t:=(X_s)_{s\in[0,t]}.$

Let us give ourself a reference probability measure $R$ on $\OO$
such that $X$ admits the $R$-semimartingale decomposition
\begin{equation}\label{eq-03}
    X=X_0+B^R+M^R,\quad R\as
\end{equation}
This means that $B^R$ is an adapted
process with bounded variation sample paths $R\as$ and $M^R$ is
a local $R$-martingale, i.e.\ there exists an increasing sequence of
stopping times $\seq\tau k$ which converges to infinity $R\as$
and
such that for each $k\ge1,$ the stopped process $t\mapsto
M^R_{t\wedge\tau_k}$ is a uniformly integrable $R$-martingale.

As a typical example, one may think of the solution to the SDE
(if it exists)
\begin{equation}\label{eq-47}
X_t=X_0+\II0t b_s(\XXX0s)\,ds+\II0t\sigma_s(\XXX0s)\, dW_s,
\quad 0\le t\le1
\end{equation}
where $W$ is a  Wiener process on $\Rd$ and $b:\iO\to\Rd$ and
$\sigma:\iO\to\mathbf{M}_{d\times d}$ are locally bounded. In
this
situation, a natural localizing sequence $\seq\tau k$ is the
sequence of exit times from the Euclidean balls of radius $k,$
$B^R_t=\int_0^tb_s(\XXX0s)\,ds$ has absolutely continuous sample
paths $R\as$ and
$M^R_t=\int_0^t\sigma_s(\XXX0s)\, dW_s$ has the quadratic
variation
\begin{equation}\label{eq-09}
    [M^R,M^R]_t=\int_0^t a_s\,ds\quad R\as
\end{equation}
where $a_t:=\sigma_t\sigma_t^*(\XXX0t)$ takes its values in the
set $\mathbf{S}_+$ of all positive semi-definite  $d\times d$
matrices.

More generally, we assume that the quadratic variation of $M^R$
is
a process which is $R\as$ equal to a random element of the set
$\mathcal{M}^{\mathrm{na}}_{\mathbf{S}_+}([0,1])$ of all bounded
measures on $\ii$ with
no atoms and taking their values in
$\mathbf{S}_+$:
\begin{equation}\label{eq-01}
[M^R,M^R](dt)=A(dt)\in
\mathcal{M}^{\mathrm{na}}_{\mathbf{S}_+}([0,1]),\quad R\as
\end{equation}
and also that
    $$
t\in[0,1]\mapsto [M^R,M^R]_t:=A([0,t])=A(t,\XXX0t;[0,t])\in
\mathbf{S}_+,\quad R\as
    $$
is an adapted process. The quadratic variation given at
\eqref{eq-01} might have an atomless singular part in addition
to
its absolutely continuous component $a_t\,dt.$ This notation is
concise: $A(dt)$ is
random and for any $\Rd$-valued processes
$\alpha, \beta,$  $\alpha_t\cdot A(dt)\beta_t$ is the
infinitesimal element of a measure on $\ii.$ In particular,
$t\mapsto \II0t A(ds)\beta_s\in\Rd,$  $t\mapsto \II0t
\alpha_s\cdot A(ds)\beta_s\in\RR$ and the process $t\mapsto \II0t
\beta_s\cdot
A(ds)\beta_s\in\RR$ is increasing.
\\
Summing up,  $R$ is a solution to the  \emph{martingale problem} $\MP(B^R,A)$. This means that the canonical process $X$ is the sum 
\eqref{eq-03} of a bounded variation adapted process $B^R$ and a
local $R$-martingale $M^R$ whose quadratic variation is specified by $A$ and
\eqref{eq-01}. We write
$$
R\in\MP(B^R,A)
$$
for short.

\begin{theorem}[Girsanov's theorem]\label{res-01}
Let $R$ and $P$ be as above, satisfying in particular the finite entropy condition \eqref{eq-02}. Then, $P$ is the law of a
semimartingale. More precisely, there exists an $\Rd$-valued
adapted process $\beta$ satisfying 
\begin{equation}\label{eq-32}
 E_P\II01\beta_t\cdot A(dt)\beta_t<\infty
\end{equation} 
and such that, defining
\begin{equation}\label{eq-18}
    \widehat{B}_t:=\II0t A(ds)\beta_s,\quad 0\le t\le 1,
\end{equation}
one obtains
\begin{equation*}
    X =X_0+B^R+\widehat{B}+M^P,\quad P\as
\end{equation*}
where $M^P$ is a local $P$-martingale such that
$[M^P,M^P]=[M^R,M^R],$ $P\as$ 
\\
In other words, $P\in\MP(B^R+\widehat{B},A).$
\end{theorem}

\begin{remarks}\
\begin{enumerate}[(a)]
\item
The process $\beta$ only needs to be defined $P\as$ (and not $R\as$) for the statement of Theorem \ref{res-01} to be meaningful. In fact, its proof only gives the ``construction'' of a process $\beta$,  $P$-almost everywhere.
\item The process $\widehat{B}$ is well defined. Indeed, by
Cauchy-Schwarz inequality, for any  $\Rd$-valued process $\xi,$
\begin{equation*}
     \Ii |\xi_t\cdot A(dt)\beta_t|
    \le  \left( \Ii \beta_t\cdot A(dt)\beta_t\right)^{1/2}
\left( \Ii \xi_t\cdot A(dt)\xi_t\right)^{1/2}\in[0,\infty],\quad P\as
\end{equation*}
Looking at $A(\omega)$ with $\omega$ fixed as a matrix of
measures, we see that $\sup\{ \Ii \xi_t\cdot A(dt)\xi_t; \xi:
|\xi_t|=1,\forall t\}$ is bounded above by the sum of the total
variations of the entries of $A$. Consequently, this supremum is
finite $P\as$ On the
other hand, as $E_P\Ii\beta_t\cdot
A(dt)\beta_t<\infty,$ we see that a fortiori $\Ii\beta_t\cdot
A(dt)\beta_t<\infty,$ $P\as$ It follows that $ \Ii
|A(dt)\beta_t|<\infty,$ $P\as$ which means that $\widehat{B}$ is
well defined.

\item When the quadratic variation is given by \eqref{eq-09},
one retrieves the standard representation
$$
    \widehat{B}_t=\II0t a_s\beta_s\,ds.
$$
It is then known that under the minimal assumption
\eqref{eq-10},
Theorem \ref{res-01} still holds true with
$$
    \Ii\beta_t\cdot a_t\beta_t\,dt<\infty,\quad P\as
$$
instead of $E_P\Ii\beta_t\cdot a_t\beta_t\,dt<\infty$ under the
assumption \eqref{eq-02}, see for instance \cite[Chp.
III]{JaShi87}.
\end{enumerate}
\end{remarks}

For any probability $Q$ on $\OO,$ let us denote $Q_0={X_0}_\#Q$
the law of the initial position $X_0$ under $Q.$

\par\medskip\noindent
\textbf{Definition} (Condition (U)). One says that $R\in \MP(B^R,A)$ satisfies
the \emph{uniqueness condition} (U) if for any probability
measure $R'$ on $\Omega$ such that the initial laws $R'_0=R_0$
are equal, $R'\ll R$ and 
$R'\in \MP(B^R,A)$, we have $R=R'.$

It is known \cite{Jac75} that if the  SDE \eqref{eq-47} admits a unique solution, for instance if the coefficients $b$ and $\sigma$ are locally Lipschitz, then its law $R$ satisfies (U).

\begin{theorem}[The density $dP/dR$]\label{res-04}
Let $R$ and $P$ be as above, satisfying in particular the finite entropy condition \eqref{eq-02}. Keeping the notation of Theorem
\ref{res-01}, we have
$$
H(P_0|R_0)+ \frac12 E_P \II0{1}\beta_t\cdot A(dt)\beta_t\le
H(P|R).
$$
If in addition it is assumed that $R$ satisfies the uniqueness
condition \emph{(U)}, then
$$
H(P_0|R_0)+  \frac12 E_P \II0{1}\beta_t\cdot A(dt)\beta_t=H(P|R)$$
and
\begin{eqnarray*}
 \frac{dP}{dR}
&=& \1_{\{\dPR>0\}} \frac{dP_0}{dR_0}(X_0)
\exp\left(\II0{1}\beta_t\cdot dM^R_t-\frac12\II0{1}\beta_t\cdot
A(dt)\beta_t
    \right)\\
&=& \1_{\{\dPR>0\}} \frac{dP_0}{dR_0}(X_0)
\exp\left(\II0{1}\beta_t\cdot
(dX_t-dB^R_t)-\frac12\II0{1}\beta_t\cdot A(dt)\beta_t
    \right).
\end{eqnarray*}
\end{theorem}
Recall that \eqref{eq-32} implies that  $\II01\beta_t\cdot A(dt)\beta_t<\infty,$ $P\as$ It follows that, although the process $\beta$ is defined only $P\as,$ the stochastic integral $\II0{1}\beta_t\cdot dM^R_t$ is meaningful on $\{\dPR>0\}.$

\subsection*{Processes with jumps in $\Rd$}

The law of a process with jumps is  a probability measure $P$ on the  canonical space
$$
    \OO=D(\ii,\Rd)
$$
of all left limited and right continuous (\cadlag) paths, endowed with its canonical filtration. We denote $X=(X_t)_{t\in\ii}$  the canonical process, $$\Delta X_t=X_t-\lg Xt$$ the jump at time $t$ and $\Rd_*:=\Rd\setminus \{0\}$ the set of all effective jumps. \\
A L\'evy kernel is a random $\sigma$-finite positive measure $$\Lb_\omega(dtdq)=\rho(dt)L_\omega(t,dq),\quad \omega\in\OO$$ on $\iR$ where $\rho$ is assumed to be a $\sigma$-finite positive \emph{atomless}  measure on $\ii.$ As a definition, any L\'evy kernel is assumed to be predictable, i.e.\ $L_\omega(t,dq)=L(X_{[0,t)}(\omega);t,dq)$ for all $t\in\ii.$\\
Let $B$ be a bounded variation \emph{continuous}  adapted process.

\begin{definition}[L\'evy kernel and martingale problem]

We say that a probability measure $P$ on $\OO$ solves the martingale problem $\MP(B,\Lb )$ if the integrability assumption 
\begin{equation}\label{eq-24a}  
E_P\IiR (|q|^2\wedge 1)\,\Lb(dtdq)<\infty
\end{equation}
holds and for any function $f$ in $\mathcal{C}^2_b(\Rd),$ the process
\begin{multline*}
f(\tilde X_t)-f (\tilde X_0)-\int_{(0,t]\times\Rd_*}[f(\lg{\tilde X}s+q)-f(\lg {\tilde X}s)-\nabla f(\lg {\tilde X}s)\cdot q]\, \1_{\{|q|\le1\}}\Lb(dsdq)\\
	-\int_{(0,t]\times\Rd_*}[f(\lg {\tilde X}s+q)-f(\lg {\tilde X}s)]\, \1_{\{|q|>1\}}\Lb(dsdq)
\end{multline*}
is a local $P$-martingale, where $\tilde X:=X-B$.  We write this 
$$
P\in\MP(B,\Lb)
$$
for short. In this case, we also say that $P$  admits the L\'evy kernel $\Lb$ and we denote this property
$$
P\in\LK(\Lb)
$$
for short.
\end{definition}

If $P\in\MP(B,\Lb),$ the canonical process is decomposed as
\begin{equation}\label{eq-22}
X=X_0+B+(\1_{\{|q|> 1\}}q)\odot \mul+ (\1_{\{|q|\le 1\}}q)\odot\mut,\quad P\as
\end{equation}
where $$\mul:= \sum_{t\in\ii; \Delta X_t\not=0}\delta_{(t,\Delta X_t)}$$ is the canonical jump measure, $\varphi(q)\odot\mul=\IiR \varphi\,d\mul=\sum_{t\in\ii; \Delta X_t\not=0}\varphi(\Delta X_t)$ and $\varphi(q)\odot\mut$ is the $P$-stochastic integral with respect to the compensated sum of jumps
$$
\mut _\omega(dtdq):= \mul_\omega(dtdq)-\Lb_\omega(dtdq).
$$

\begin{definition}[Class $\mathcal{H}_{p,r}(P,\Lb )$] Let $P$ be a probability measure on $\OO$ and $\Lb$ a L\' evy kernel such that $P\in \LK(\Lb).$
We say that a predictable integrand $h_\omega(t,q)$ is in the class $\mathcal{H}_{p,r}(P,\Lb )$ if $E_P\IiR\1_{\{|q|\le1\}}|h_t(q)|^p\,\Lb(dtdq)<\infty$ and $E_P\IiR\1_{\{|q|>1\}}|h_t(q)|^r\,\Lb(dtdq)<\infty.$
\\
We denote $\mathcal{H}_{p,p}(P,\Lb )=\mathcal{H}_{p}(P,\Lb )$.
\end{definition}

We take our reference law  $R$  such that
$$
	R\in\MP(B^R,\Lb )
$$ 
for some adapted continuous bounded variation process $B^R.$
The integrability assumption \eqref{eq-24a} 
means that the integrand $|q|$ is in $\mathcal{H}_{2,0}(R,\Lb ).$ This will be always assumed in the future.
We introduce the function
$$
	\theta(a)=\log \mathbb{E} e^{a(N-1)}=e^a-a-1,\quad a\in\RR.
$$
where $N$ is a Poisson(1) random variable. Its convex conjugate is
$$
	\theta^*(b)=\left\{
	\begin{array}{ll}
(b+1)\log(b+1)-1& \textrm{if }b>-1\\
1&\textrm{if }b=-1\\
\infty&\textrm{otherwise}
\end{array}
\right.
,\quad b\in\RR
$$
Note that $\theta$ and $\theta^*$ are
respectively equivalent to $a^2/2$ and $b^2/2$ near zero.

\begin{theorem}[Girsanov's theorem. The jump
case]\label{res-01b}
Let $R$ and $P$ be as above: $R\in\MP(B^R,\Lb)$ and $H(P|R)<\infty.$  Then, there exists a unique
predictable nonnegative process 
$\ell:\OO\times\ii\times\Rd_*\to[0,\infty)$ satisfying
\begin{equation}\label{eq-29}
	 E_P\IiR \theta^*(|\ell-1|)\, d\Lb<\infty,
\end{equation}
such that $P\in\MP(B^R+\widehat B^\ell,\ell\Lb)$ 
where
$$
\widehat B^\ell_t=\int_{[0,t]\times\Rd_*}\1_{\{|q|\le1\}} (\ell_s(q)-1)q\, \Lb(dsdq),\quad t\in\ii
$$
is well-defined $P\as$
\end{theorem}

It will appear that, in several respects, $\log\ell$ is analogous to $\beta$ in Theorem \ref{res-01}.
 Again, $\ell$ only needs to be defined $P\as$ and not $R\as$ for the statement of Theorem \ref{res-01b} to be meaningful. And indeed, its proof will only provide a $P\as$-construction  of $\ell.$

\begin{corollary}\label{res-08}
Suppose that in addition to the assumptions of Theorem \ref{res-01b}, there exist some $a_o,b_o,c_o>0$ such that 
\begin{equation}\label{eq-35}
E_R\exp\left(a_o \IiR \1_{\{|q|>c_o\}} e^{b_o|q|}\,\Lb(dtdq)\right)<\infty.
\end{equation}
It follows immediately that $\1_{\{|q|>c_o\}}|q|$ is $R\otimes\Lb$-integrable so that the stochastic integral $q\odot\mut$ is well-defined $R\as$ and we are allowed to rewrite \eqref{eq-22} as
$$
X =X_0+B+q\odot\mut,\quad R\as,
$$
for some adapted continuous bounded variation process $B.$ 
\\
Then, there exists a unique predictable nonnegative process 
$\ell:\OO\times\ii\times\Rd_*\to[0,\infty)$ satisfying \eqref{eq-29} such that
$$
X =X_0+B+\overline B^\ell+q\odot\muh,\quad P\as,
$$
where
$$
	\overline B_t^\ell=\int_{[0,t]\times\Rd_*} (\ell_s(q)-1)q\, \Lb(dsdq),\quad t\in\ii
$$
is  well-defined $P\as$ and the $P$-stochastic integral $q\odot \muh $  with respect to the L\'evy kernel $\ell\Lb$  is a local $P$-martingale. 
\end{corollary}

\begin{remarks}\ 
\begin{enumerate}[(a)]
\item
The energy estimate  \eqref{eq-29} is equivalent to: 
$\1_{\{0\le\ell\le2\}}(\ell-1)^2$ and
$\1_{\{\ell\ge2\}}\ell\log\ell$ are integrable with respect to
$P\otimes\Lb .$
\item
Together with \eqref{eq-29}, \eqref{eq-35} implies that the integral for $\overline B^\ell$ is well-defined since 
\begin{equation}\label{eq-36}
E_P\int_{[0,1]\times\Rd_*} (\ell_t(q)-1)|q|\, \Lb(dtdq)<\infty.
\end{equation}
\end{enumerate}
\end{remarks}

In the present context of processes with jumps, the uniqueness condition (U) becomes:

\par\medskip\noindent
\textbf{Definition} (Condition (U)). One says that $R\in \MP(B^R,\Lb)$ satisfies
the \emph{uniqueness condition} (U) if for any probability
measure $R'$ on $\Omega$ such that the initial laws $R'_0=R_0$
are equal, $R'\ll R$ and  $R'\in \MP(B^R,\Lb)$, we have $R=R'.$

\begin{theorem}[The density $dP/dR$]\label{res-04b}
Suppose that $R$ and $P$ verify $R\in\MP(B,\Lb)$ and $H(P|R)<\infty.$ With $\ell$ given at Theorem \ref{res-01b}, we have
$$
H(P_0|R_0)+ E_P\IiR (\ell\log\ell-\ell+1)\,d\Lb
\le H(P|R)
$$
with the convention $0\log 0-0+1=1.$
\\
If in addition it is assumed that $R$ satisfies the uniqueness
condition \emph{(U)}, then
$$
H(P_0|R_0)+  E_P\IiR (\ell\log\ell-\ell+1)\,d\Lb=H(P|R)$$
and
\begin{equation}\label{eq-42}
 \frac{dP}{dR}
= \1_{\{\dPR>0\}}\frac{dP_0}{dR_0}(X_0)\ 
\widetilde{\exp}\left(\log\ell\odot \mut_1 -\IiR \theta(\log\ell) \,d\Lb
\right).
\end{equation}
In formula \eqref{eq-42}, $\widetilde{\exp}$ indicates a shorthand for the rigorous following expression
\begin{equation}\label{eq-43}
\left\{
\begin{array}{rcl}
\displaystyle{\frac{dP}{dR}}&=&\displaystyle{\frac{dP_0}{dR_0}(X_0) Z^+Z^-}\quad \textrm{with}\\
Z^+&=& \displaystyle{\1_{\{\dPR>0\}}\exp\left([\1_{\{\ell\ge 1/2\}}\log\ell]\odot \mut_1 -\int_{\{\ell\ge 1/2\}}(\ell-\log\ell-1) d\Lb\right)}\\
Z^-&=&\displaystyle{\1_{\{\dPR>0,\tau^-=\infty\}}
\exp\left(-\int_{\{0\le\ell< 1/2\}}[\ell-1] d\Lb\right)}
\displaystyle{
\prod_{0\le t\le1;0<\ell(t,\Delta X_t)<1/2}\ell(t,\Delta X_t)}
\end{array}
\right.
\end{equation}
where 
$$
\tau^-:=\sup_{n\ge1}\inf\left\{t\in\ii;\ell(t,\Delta X_t)\le 1/n \right\}
\in\ii\cup\{\infty\},
$$
 with the convention that $\inf\emptyset=\infty$.
\end{theorem}

Note that although $\ell$ is only defined $P\as,$ $Z^+$, $Z^-$ and $\tau^-$ are meaningful thanks to the prefactors $\1_{\{\dPR>0\}}.$ 

\begin{remarks}\ \begin{enumerate}[(a)]
\item
Because of \eqref{eq-29}, the integral  $\int_{\{\ell\ge 1/2\}}(\ell-\log\ell-1) d\Lb$ inside $Z^+$ is finite $P\as$
\item
Similarly, the product $\prod_{0\le t\le1;0<\ell(t,\Delta X_t)<1/2}\ell(t,\Delta X_t)$ doesn't vanish $P\as$ because it is proved at Lemma \ref{res-09} that $P(\tau^-=\infty)=1.$
\item
Note that this product  is well-defined in $[0,1]$ since it contains $P\as$ at most countably many terms in $(0,1/2].$ But, if it contains infinitely many such terms, it vanishes. Therefore, it contains $P\as$ finitely many terms.
\item 
Since $\inf\left\{t\in\ii;\ell(t,\Delta X_t)=0 \right\}\ge \tau^-,$
if $\ell(t,\Delta X_t)=0$ for some $t\in\ii,$ then $\dPR=0.$ Therefore, $\ell>0,$ $P\as$
\item
If $\1_{\{\ell\ge1/2\}}\log\ell$ is $R\otimes\Lb$-integrable, an alternate expression of $\dPR$ is
$$
\dPR=\1_{\{\dPR>0,\tau^-=\infty\}} \frac{dP_0}{dR_0}(X_0)\exp\left(-\IiR (\ell-1)\, d\Lb\right)\prod_{0\le t\le1}\ell(t,\Delta X_t).
$$
\item
If $\1_{\{0\le\ell< 1/2\}}\log\ell$ is not $R\otimes\Lb$-integrable, then $\log\ell\odot \mut_1$ is undefined and \eqref{eq-42} with $\exp$ instead of $\widetilde\exp$ is meaningless and must be replaced by \eqref{eq-43}.
\item
On the other hand, if $\ell>0,$ $R\as$ and $E_R\IiR\theta(\log\ell)\,d\Lb<\infty,$ then  \eqref{eq-42} gives the rigorous expression for $\dPR$ with $\exp$ instead of $\widetilde\exp$.
\end{enumerate}\end{remarks}
For more details about the relationship between \eqref{eq-42} and \eqref{eq-43}, see the discussion below Proposition \ref{res-03b} at the Appendix.

\section{Variational representations of the relative entropy}

Theorems \ref{res-01} and \ref{res-01b}'s proofs rely on some variational
representation of the relative entropy which is stated and proved below.

\begin{proposition}[Variational representations of the relative
entropy]\label{res-02}
Let $R$ be a probability measure on some space $\OO.$
\begin{enumerate}
    \item For any signed bounded measure $P$ on
    $\OO,$ we have
    \begin{eqnarray*}
&& \sup\left\{\int u\,dP-\log\int e^u\,dR; u\textrm{ bounded
measurable}\right\}\\
   &=& \sup\left\{\int u\,dP-\log\int e^u\,dR; u\in
   L^\infty(P)\right\}\\
   &=& \left\{
    \begin{array}{ll}
H(P|R)\in[0,\infty], & \textrm{if $P$ is a probability measure
and }P\ll R \\
      \infty, & \textrm{otherwise}. \\
    \end{array}
   \right.
\end{eqnarray*}
    \item For any probability measure $P$ on $\OO$
such that $P\ll R,$
\begin{equation*}
    H(P|R) =\sup\left\{\int u\,dP-\log\int e^u\,dR; u: \int
    e^{u}\,dR<\infty, \int u_-\,dP<\infty
    \right\}\in[0,\infty]
\end{equation*}
where $u$ is measurable, $u_-=(-u)\vee 0$ and $\int
u\,dP\in(-\infty,\infty]$ is well-defined for all $u$ such that
$\int u_-\,dP<\infty.$
\item If in addition it is known that the probability measure
$P$ satisfies $H(P|R)<\infty,$ then
any measurable function $u$ such that $\int e^u\,dR<\infty$
verifies $u\in
L^1(P)$ and we have
\begin{equation}\label{eq-11}
H(P|R) =\sup\left\{\int u\,dP-\log\int e^u\,dR; u: \int
e^u\,dR<\infty\right\}.
\end{equation}
In this formula the supremum is taken over all measurable
functions  $u:\Omega\to [-\infty,\infty),$ possibly taking the
value $-\infty$ with the convention $e^{-\infty}=0.$ On the
other
hand, the supremum is attained at
$u^*=\1_{\{dP/dR>0\}}\log(dP/dR)-\1_{\{dP/dR=0\}}\infty,$
corresponding to $e^{u^*}=dP/dR.$ If $R$ is not a Dirac measure,
this supremum is
uniquely attained.
\end{enumerate}
\end{proposition}

\begin{proof} Let us first  \textbf{prove (1)}.
Denote
$$
\kappa:=\sup\left\{\int u\,dP-\log\int e^u\,dR; u\textrm{ bounded
measurable}\right\}
$$
and
$$
\kappa':=\sup\left\{\int u\,dP-\log\int e^u\,dR; u\in
L^\infty(P)\right\}.
$$
Let us show that when $P$ in not positive, i.e.\ $P_-\not=0$,
then
$\kappa=\kappa'=\infty.$ Let $A$ be a measurable subset such that
$P_+(A)=0$
and $P_-(A)>0.$ Then, choosing $u_a=-a\1_A$ with $a>0,$ we see
that
    $\kappa\ge\Lim a (\int u_a\,dP-\log\int e^{u_a}\,dR)
=\Lim a (aP_-(A)-\log[1+(e^{-a}-1)R(A)])=+\infty.$ Similarly for
$\kappa'.$
\\
Now, suppose that $P$ is a positive measure such that
$P(\Omega)\not=1.$ Let us show that $\kappa=\kappa'=\infty.$ Considering
the
constant functions $u\equiv a\in\RR,$ we see that
   $\int a\,dP-\log\int e^a\,dR=a(P(\Omega)-1).$ Letting $|a|$
tend to infinity, we obtain $\kappa\ge
\sup_a\left\{a(P(\Omega)-1)\right\}=\infty.$ And similarly for
$\kappa'.$
\\
Let us show that, if the probability measure $P$ is not
absolutely
continuous with respect to $R,$ then $\kappa=\kappa'=\infty.$ For such a
$P,$ there is a measurable set $A$ such that $P(A)>0$ and
$R(A)=0.$ Considering the functions $u=a\1_A,$ we see that
    $\kappa\ge \sup_a\{aP(A)-0\}=\infty,$ and similarly for $\kappa'.$

\emph{From now on, $P$ is a probability measure such that $P\ll
R.$}
\\
Let us have a look at the first equality of assertion (1). Since
the bounded functions
are in $L^\infty(P)$, it is clear that
$\kappa\le
\kappa'.$ On the other hand, we also have $\kappa'\le \kappa.$ Indeed, one can
write any $u$ in $L^\infty(P)$ as
$u=\1_{\{dP/dR>0\}}v+\1_{\{dP/dR=0\}}w$ where $v$ is bounded and
$w$ is unspecified. But,\begin{eqnarray*}
   && \int u\,dP-\log\int e^u\,dR\\
    &=&\int v\,dP-\log\left(\int \1_{\{dP/dR>0\}}e^v\,dR+\int
    \1_{\{dP/dR=0\}}e^w\,dR\right)\\
    &\le& \int v\,dP-\log\int \1_{\{dP/dR>0\}}e^v\,dR\\
    &=&\Lim n\left(
    \int u_n\,dP-\log\int e^{u_n}\,dR\right)
\end{eqnarray*}
with $u_n:=\1_{\{dP/dR>0\}}v-n\1_{\{dP/dR=0\}}$. As the
functions
$u_n$ are bounded, we see that $\kappa'\le \kappa.$
\\
To prove (1), it remains to show that
\begin{equation}\label{eq-20}
    \kappa=H(P|R).
\end{equation}

We begin \textbf{proving (2)}. The identity \eqref{eq-20} will
appear as a step. The
remainder of the proof relies on Fenchel's inequality for the
convex $\theta(a):=a\log
a-a+1$ and on its
equality case. This inequality is
\begin{equation}\label{eq-04}
     ab\le (a\log a-a+1)+(e^b-1)=\theta(a)+(e^b-1)
\end{equation}
for all $a\in[0,\infty),$ $b\in[-\infty,\infty),$  with the
conventions $0\log 0=0,$ $e^{-\infty}=0$ and $-\infty\times 0=0$
which are legitimated by
limiting procedures. The equality is
realized if and only if $a=e^b.$
\\
We denote $Z:=\frac{dP}{dR}$ for a simpler notation. Taking
$a=Z(\omega),$ $b=u(\omega)$ and integrating with respect to $R$
leads us to
$$
\int u\,dP\le \int\theta(Z)\,dR+\int(e^u-1)\,dR=
H(P|R)+\int(e^u-1)\,dR,
$$
whose terms are meaningful provided that they are understood in
$(-\infty,\infty],$ as
soon as $\int u_-\,dP<\infty.$ Formally,
the equality case corresponds to $Z=e^u.$ By the monotone
convergence theorem, it can be approximated by the sequence
$u_n=
\log(Z\vee e^{-n}),$ as $n$ tends to infinity. This gives us
\begin{equation}\label{eq-19}
    H(P|R)=\sup\left\{\int u\,dP-\int (e^u-1)\,dR; u: \int
e^{u}\,dR<\infty,\inf u>-\infty \right\},
\end{equation}
which in turn implies that
\begin{equation}\label{eq-05}
    H(P|R)=\sup\left\{\int u\,dP-\int (e^u-1)\,dR; u: \int
e^{u}\,dR<\infty,\int u_-\,dP<\infty \right\},
\end{equation}
since the integral $\int \log Z\,dP=\int
\theta(Z)\,dR\in[0,\infty]$ is well-defined.
\\
Now, let us take advantage of the unit mass of $P:$
$$
\int (u+b)\,dP-\int (e^{(u+b)}-1)\,dR=\int u\,dP-e^b\int e^u\,dR
+b+1,\quad \forall b\in
\RR.
$$
Thanks to the elementary identity $\log
a=\inf_{b\in\RR}\{ae^b-b-1\},$ we obtain
$$
\sup_{b\in\RR}\left\{\int
(u+b)\,dP-\int(e^{(u+b)}-1)\,dR\right\}=\int u\,dP-\log\int
e^u\,dR.
$$
Hence,
\begin{eqnarray*}
&&\sup\left\{\int u\,dP-\int (e^u-1)\,dR; u: \int
e^{u}\,dR<\infty,\int u_-\,dP<\infty
  \right\}\\
&=& \sup\left\{\int (u+b)\,dP-\int (e^{(u+b)}-1)\,dR;b\in\RR, u:\int e^{u}\,dR<\infty,\int u_-\,dP<\infty \right\}\\
&=& \sup\left\{\int u\,dP-\log\int e^u\,dR; u: \int
e^{u}\,dR<\infty,\int u_-\,dP<\infty
  \right\},
\end{eqnarray*}
With \eqref{eq-05}, this proves assertion (2).
\\
But a similar reasoning, starting from \eqref{eq-19} instead of
\eqref{eq-05}, leads us to the similar following conclusion
$$
H(P|R)=\sup\left\{\int u\,dP-\log\int e^u\,dR; u: \int
e^{u}\,dR<\infty,\inf u>-\infty\right\}.
$$
Considering the functions $u\wedge n$ with $\inf u>-\infty$ and
letting $n$ tend to infinity, this leads us to \eqref{eq-20} and
proves assertion (1).

\textbf{Let us prove (3)}. Suppose that $H(P|R)<\infty.$ With
the
inequality \eqref{eq-04}, we obtain
    $
    |u|Z=|uZ|\le \theta(Z)+ e^u.
    $
Therefore, if  $\int e^u\,dR<\infty,$ then
\begin{equation*}
    E_P|u|=E_R(|u|Z)\le
    E_R\theta(Z)+E_Re^u=H(P|R)+E_Re^u<\infty.
\end{equation*}
This means that $u$ is $P$-integrable and shows \eqref{eq-11}.

We check directly the equality case. The uniqueness of its
realization comes from the \emph{strict} concavity of the
function
$u\mapsto \int u\,dP-\log\int e^u\,dR.$ One shows the strict
convexity of $u\mapsto \log\int e^u\,dR$ by means of H\"older's
inequality. But it is
also possible to come back to the
representation \eqref{eq-05} which, with the same reasoning as
above, leads us to
\begin{equation*}
H(P|R)=\sup\left\{\int u\,dP-\int (e^u-1)\,dR; u:\int
e^u\,dR<\infty\right\}.
\end{equation*}
Then, one directly reads the strict convexity of $u\mapsto\int
(e^u-1)\,dR$.
\end{proof}

\section{Proof of Theorem \ref{res-01}}

For the proof of Theorem \ref{res-01} we need to exhibit a large
enough family of
exponential supermartingales.

\begin{lemma}[Exponential supermartingales]\label{res-03}
Let $M$ be a  local martingale, then
$$
Z^M_t=\exp\left(M_t-\frac12 [M,M]_t\right),\quad 0\le t\le 1,
$$
is also a local martingale and a supermartingale. In particular,
$0\le E_RZ^M_1\le1.$
\end{lemma}

\begin{proof}
Recall It\^o's formula
$$
df(Y_t)=f'(Y_t)\,dY_t+\frac12f''(Y_t)\,d[Y,Y]_t
$$
which is valid for any $\mathcal{C}^2$ function $f$  and any continuous
semimartingale $Y$. Applying it to $Y_t=M_t-\frac12 [M,M]_t$ and
$f(y)=e^y,$ we obtain
$$
dZ^M_t=Z^M_t\left(dM_t-\frac12 d[M,M]_t+\frac12 d[M,M]_t\right)
    =Z^M_t\,dM_t
$$
which proves that $Z^M$ is a local martingale. Since $Z^M\ge0,$
Fatou's lemma applied to the localized sequence
$Z^M_{t\wedge\tau_k}$ as $k$ tends to infinity tells us that
$Z^M$
is a $R$-supermartingale, with $\seq\tau{k}$ an increasing
sequence of stopping times which tends almost surely to infinity
and localizes the local
martingale $M$. In particular,
$E(Z^M_1)\le E(Z^M_0)=1.$
\end{proof}

The standard notation for the supermartingale of Lemma
\ref{res-03} is
$$
    \mathcal{E}(M):=\exp\left(M-\frac12 [M,M]\right).
$$
We are now ready for the proof of Theorem \ref{res-01}.

\begin{proof}[Proof of Theorem \ref{res-01}]
We start with some useful notation. Let $Q$ be a probability
measure on $\OO;$ later we shall take $Q=R$ or $Q=P.$ For any
measurable function $g$ on $\iO,$ let us denote
    $$
    (g,g)_A(\omega):=\Ii g_t(\omega)\cdot
    A_t(\omega;dt)g_t(\omega)\in[0,\infty]
    $$
and introduce the function space
$$
\mathcal{G}(Q):=\left\{g:\iO\to \Rd; g\textrm{ measurable}, E_Q
(g,g)_A <\infty
    \right\}
$$
endowed with the seminorm $\|g\|_{\mathcal{G}(Q)}:=(E_Q
(g,g)_A)^{1/2}.$ Identifying the functions with their
equivalence
classes when factorizing by the kernel of this seminorm, turns
$\mathcal{G}(Q)$ into a Hilbert space. These equivalence classes
are called
$\mathcal{G}(Q)$-classes and with some abuse, we say
that two elements of the same class are equal
$\mathcal{G}(Q)$-almost everywhere. The relevant space of
integrands for the stochastic integral is
$$
     \mathcal{H}^Q:=\{h\in \mathcal{G}(Q); h\textrm{
     adapted}\}.
$$
Identity \eqref{eq-03} says that $M^R=X-X_0-B^R$ is a local
$R$-martingale. For all $h\in\mathcal{H}^R,$ let us denote the
stochastic integral
$$
    h\cdot M^R_t:=\int_0^t h_s\,dM^R_s,\quad t\in\ii.
$$
By Lemma \ref{res-03},
    $
0<E_RZ^{h\cdot M^R}_1\le1$ for all $h\in\mathcal{H}^R
    $
and  because of \eqref{eq-11}, for any probability measure $P$
such that $H(P|R)<\infty,$ we have
\begin{equation}\label{eq-08}
E_P\left(h\cdot M^R_1-\frac12 [h\cdot M^R,h\cdot M^R]_1\right)\le H(P|R),\ \forall
h\in\mathcal{H}^R.
\end{equation}
Note that, as $P\ll R,$ $h\cdot M^R_1$ and $[h\cdot M^R,h\cdot M^R]_1$  which are
defined $R\as,$ are a fortiori defined $P\as$  With \eqref{eq-01} and \eqref{eq-08}, we see that
\begin{equation}\label{eq-06}
    E_P(h\cdot M^R)\le H(P|R)+\frac12 E_P (h,h)_A,\quad \forall
    h\in  \mathcal{G}(P)\cap \mathcal{H}^R.
\end{equation}
The notation $\mathcal{G}(P)\cap \mathcal{H}^R$ is a little bit
improper. Indeed,
$\mathcal{G}(P)$ is a set of equivalence
classes
with respect to the equality $\mathcal{G}(P)\ae$, while
$\mathcal{H}^R$ is a set of $\mathcal{G}(R)$-classes. But since
$P\ll R,$ keeping in mind that any $\mathcal{G}(P)$-class is the
union of some
$\mathcal{G}(R)$-classes, one can interpret
$\mathcal{G}(P)\cap\mathcal{G}(R)$ as a set of
$\mathcal{G}(P)$-classes. It is also clear that
$\mathcal{G}(P)\cap\mathcal{H}^R=\mathcal{H}^P\cap\mathcal{H}^R$
which is a set of
$\mathcal{G}(P)$-classes. Considering $-h$ in
\eqref{eq-06}, we obtain for all $\lambda>0,$
\begin{equation*}
    \left| E_P \left(\frac h\lambda\cdot M^R\right)\right|
    \le H(P|R)+\frac1{2\lambda^2}E_P  (h,h)_A,\quad \forall
    h\in \mathcal{H}^P\cap\mathcal{H}^R.
\end{equation*}
Let
$$
    S :=\left\{h:[0,1]\times\Omega\to \mathbb{R}^d;
h=\sum_{i=1}^k h_i\1_{\rrbracket S_i,T_i\rrbracket} \right\}$$
denote the set of all \emph{simple} adapted processes $h$ where
$k$ is finite and for all $i,$ $h_i\in\mathbb{R}^d$ and $S_i\le
T_i$ are stopping times.
As $S \subset
\mathcal{H}^P\cap\mathcal{H}^R,$ taking
$\lambda=\|h\|_{\mathcal{G}(P)}$ in previous inequality, we
obtain
the keystone of the proof:
\begin{equation*}
    | E_P(h\cdot M^R)|
    \le [H(P|R)+1/2 ]\,\|h\|_{\mathcal{G}(P)},
    \quad \forall h\in S .
\end{equation*}
This estimate still holds when $\|h\|_{\mathcal{H}(P)}=0.$
Indeed,
for all real $\alpha$, by \eqref{eq-06} we see that $\alpha
E_P(h\cdot M^R)\le H(P|R)+\alpha^2/2\ E_P (h,h)_A=H(P|R).$
Letting
$|\alpha|$ tend to infinity, it follows that $E_P(h\cdot
M^R)=0.$

Under the assumption that $H(P|R)$ is finite, this means that
$h\mapsto h\cdot M^R$ is continuous on $S $ with
respect
to the Hilbert topology of $\mathcal{H}^P.$ As $S $ is
dense in $\mathcal{H}^P,$ this linear form extends uniquely as a continuous linear form on $\mathcal{H}^P.$ It also appears that
this extension is again a stochastic integral \emph{with respect
to $P.$} We still denote
this extension by $h\cdot M^R.$
\\
As $h\mapsto h\cdot M^R$ is a continuous linear form on
$\mathcal{H}^P,$ we know by the Riesz representation theorem
that
there exists a unique $\beta\in\mathcal{H}^P$ such that
\begin{equation*}
    E_P(h\cdot M^R)=E_P  (\beta, h)_A,
    \quad \forall h\in \mathcal{H}^P.
\end{equation*}
In other words,
\begin{equation*}
    E_P\Ii h_t\,dM^P_t=0,\quad\forall h\in \mathcal{H}^P
\end{equation*}
where
$$
M^P_t:=M^R_t-\II0t A(ds)\beta_s=X_t-X_0-B^R_t-\widehat{B}_t,
$$
which means that $M^P$ is a local $P$-martingale.
\end{proof}

\section{Proof of Theorem \ref{res-04}}

It relies on a transfer result which is stated below at Lemma
\ref{res-05}. But we first need to introduce its framework and
some additional notation. Let $P$ be a probability measure on
$\OO$ such that $[X,X]=A,$ $P\as$ and
$$
X=X_0+B+M^P,\quad P\as,
$$
where $B$ is a bounded variation process and $M^P$ is a local
$P$-martingale. Let $\gamma$ be an adapted process such that
$\Ii
\gamma_t\cdot A(dt)\gamma_t<\infty,$ $P\as$ We define
$$
Z_t=\exp\left(\II0t\gamma_s\,dM^P_s-\frac12 \II0t \gamma_s\cdot
A(ds)\gamma_s\right),\quad 0\le t\le 1
$$
and for all $k\ge1,$
$$
    \sigma^k:=\inf\left\{t\in[0,1]; \II0t \gamma_s\cdot
    A(ds)\gamma_s\ge k \right\}\in[0,1]\cup \{\infty\},
$$
with the convention $\inf\emptyset=\infty.$
\\
We use the standard notation $Y^\tau_t=Y_{\tau\wedge t}$ for the
process $Y$ stopped at a
stopping time $\tau.$ For all $k,$
$P^k:={X^{\sigma_k}}_\#P$ is the push-forward of $P$ with
respect
to the stopping procedure $X^{\sigma_k}.$ Note that $P^k$ and
$P$
match on the $\sigma$-field which is generated by
$\XXX0{\sigma_k}.$

\begin{lemma}\label{res-05}
Let  $P$ and $\gamma$ as above. Then, for all $k\ge1,$
$Z^{\sigma_k}$ is a genuine $P$-martingale and the measure
$$
    Q^k:=Z^{\sigma_k}_1 P^k
$$
is a probability measure on $\Omega$ which satisfies $Q^k\in\MP(\widehat{B}^{\sigma_k},A^{\sigma_k})$
where  $\widehat{B}^{\sigma_k}_t=\II0{t\wedge\sigma_k}
A(ds)\gamma_s$ and $M^k$ is a local $Q^k$-martingale.
\end{lemma}

\begin{proof}
Let us first show that $Z^{\sigma_k}$ is a
$P^k$-martingale\footnote{It is a direct consequence of
Novikov's
criterion, but we prefer presenting an elementary proof  which
will be a guideline for a similar result with jump processes.}.
The local martingale  $Z^{\sigma_k}$ is of the form
$Z^{\sigma_k}=\mathcal{E}(N):=\exp(N-\frac12 [N,N])$ with $N$ a
local $P^k$-martingale such that $[N,N]\le k,$ $P^k\as$ For all
$p\ge0,$ since $\mathcal{E}(N)^p=\exp(pN-\frac{p}2[N,N])$ and
$\mathcal{E}(pN)=\exp(pN-\frac{p^2}2[N,N])\ge
e^{pN}e^{-kp^2/2},$
we obtain
$$
    \mathcal{E}(N)^p\le e^{pN}\le e^{kp^2/2}\mathcal{E}(pN).
$$
As a nonnegative local martingale, $\mathcal{E}(pN)$ is a
supermartingale. We deduce from this that
$E_{P^k}\mathcal{E}(pN)\le 1$ and
$$
E_{P^k}\mathcal{E}(N)^p\le e^{kp^2/2}E_{P^k}\mathcal{E}(pN)\le
e^{kp^2/2}<\infty.
$$
Choosing  $p>1,$ it follows that $\mathcal{E}(N)$ is uniformly
integrable. In particular, this implies that
$$
    E_{P^k}\mathcal{E}(N)_1=E_{P^k}\mathcal{E}(N)_0=1
$$
and proves that $Q^k$ is a probability measure.
\\
Suppose now that the supermartingale $\mathcal{E}(N)$ is not a
martingale. This implies that there exists $0\le t< 1$ such that
on a subset with
positive measure,
$E_{P^k}(\mathcal{E}(N)_1\mid\XXX0t)<\mathcal{E}(N)_t.$
Integrating, we
get
$1=E_{P^k}\mathcal{E}(N)_1<E_{P^k}\mathcal{E}(N)_t,$ which
contradicts $E_{P^k}\mathcal{E}(N)_s\le
E_{P^k}\mathcal{E}(N)_0=1, \forall s$: a
consequence of the supermartingale property of
$\mathcal{E}(N)$. Therefore, $\mathcal{E}(N)$ is a genuine
$P^k$-martingale.

Let us fix $k\ge1$ and show that $Q^k$ is a solution to $\MP(\widehat{B}^{\sigma_k},A^{\sigma_k})$. First of all, as it is assumed
that
$[X,X]=A,$ $P\as,$ we obtain $[X,X]=A^{\sigma_k},$ $P^k\as$ With
$Q^k\ll P^k,$ this
implies that $[X,X]=A^{\sigma_k},$ $Q^k\as$
\\
Now, we check 
\begin{equation}\label{eq-12}
    X=X_0+B^{\sigma_k}+\widehat{B}^{\sigma_k}+M^k
\end{equation}
where $M^k$ is a $Q^k$-martingale.
Let $\tau$ be a stopping time and
denote $F_t=\xi\cdot X^\tau_t$ with $\xi\in\Rd.$ The martingale
$Z^{\sigma_k}$ is the stochastic exponential $\mathcal{E}(N)$
of $N_t=\II0t
\1_{[0,\sigma_k]}(s)\gamma_s\cdot dM^P_s.$ Hence,
denoting $Z=Z^{\sigma_k},$ we have
    $dZ_t=Z_t \1_{[0,\sigma_k]}(t)\gamma_t\cdot dM^P_t,$
    $dF_t=\1_{[0,\tau]}(t)\xi\cdot(dB_t+dM^P_t)$ and
$d[Z,F]_t=Z_t\1_{[0,\tau\wedge\sigma_k]}(t)\xi\cdot
A(dt)\gamma_t,$ $P^k\as$
Consequently,
\begin{eqnarray*}
  E_{Q^k}[\xi\cdot(X_\tau-X_0)]
  &\overset{(a)}=& E_{P^k}[Z_\tau F_\tau -Z_0F_0] \\
&\overset{(b)}=& E_{P^k}\left[\II0\tau (F_t\,dZ_t +
Z_tdF_t+d[Z,F]_t)\right] \\
&=& E_{P^k}\left[\II0\tau F_t\,dZ_t + \II0\tau
Z_t\xi\cdot(dB_t+dM^P_t)
  +\II0\tau Z_t\xi\cdot A(dt)\gamma_t\right] \\
&\overset{(c)}=& E_{P^k}\left[\II0\tau Z_t\xi\cdot dB_t+\II0\tau Z_t\xi\cdot
  A(dt)\gamma_t\right]\\
  &\overset{(d)}=& E_{Q^k}\left[ \xi\cdot \II0\tau (dB_t+
  A(dt)\gamma_t)\right].
\end{eqnarray*}
In order that all the above terms are meaningful, we choose
$\tau$
such that it localizes $F,$ $B,$ $M^P$ and $\xi\cdot A\gamma.$
This is possible, taking for any $n\ge1,$
$\tau\le\tau_n=\min(\tau^F_n,\tau^B_n,\tau^M_n,\tau^\gamma_n)$
where
    $\tau^F_n=\inf\{t\in[0,1]; |X_t|\ge n\},$
    $\tau^B_n=\inf\{t\in[0,1]; \II0t |dB_s|\ge n\},$
$\tau^\gamma_n=\inf\{t\in[0,1];\II0t \gamma_s\cdot
A(ds)\gamma_s\ge n\},$
and $\tau^M_n$ is a localizing sequence of the local martingale
$M^P.$ We have
\begin{equation}\label{eq-13}
     \Lim n \tau_n=\infty,\quad P^k\as
\end{equation}
We used the definition of $Q^k$ and the martingale property of
$Z$
at (a) and (d), (b) is It\^o's formula and (c) relies on the
martingale property of $Z$ and $(M^P)^\tau$. Finally, taking
$\tau=\varsigma\wedge\tau_n,$ we see that for any stopping time
$\varsigma,$ any $n\ge1$ and any $\xi\in\Rd$
$$
    E_{Q^k}[\xi\cdot(X_\varsigma^{\tau_n}-X_0^{\tau_n})]
    = E_{Q^k}\left[ \xi\cdot \II0{\varsigma\wedge\tau_n} (dB_t+
  A(dt)\gamma_t)\right].
$$
Taking \eqref{eq-13} into account, this means that
$X-X_0-B-\widehat{B}$ is a local $Q^k$-martingale. We conclude
remarking that for any process $Y,$ we have $Y=Y^{\sigma^k},$
$Q^k\as$ This leads us to \eqref{eq-12}.
\end{proof}

Let us denote $P^\tau={X^\tau}_\#P$ the law under $P$ of the
process $X^\tau$ which is stopped at the stopping time $\tau$.

\begin{lemma}\label{res-06}
If $R$ fulfills the condition (U), then for any stopping time
$\tau,$ $R^\tau$ also fulfills it.
\end{lemma}

\begin{proof}
Let us fix the stopping time $\tau.$ Our assumption on $R$
implies
that
$$
X=X_0+B+M,\quad R^\tau\as
$$
where $M=M^R$ is a local $R$-martingale and we denote $B=B^R$.
Let $Q\ll R^\tau$ be given
such that $Q_0=R_0$ and
$$
X=X_0+B+M^Q,\quad Q\as
$$
where $M^Q$ is a local $Q$-martingale. We wish to show that
$Q=R^\tau.$
\\
The disintegration
$$
R=R_{[0,\tau]}\otimes R(\cdot\mid\XXX0\tau)
$$
means that for any bounded measurable function $F$ on $\OO,$
denoting $F=F(X)=F(\XXX0\tau, \xXX\tau1),$
$$
E_R(F)=\IO E_R[F(\eta,\xXX\tau1)\mid
\XXX0\tau=\eta]\,R_{[0,\tau]}(d\eta).
$$
Similarly, we introduce the probability measure
$$
    R':=Q_{[0,\tau]}\otimes R(\cdot\mid\XXX0\tau).
$$
To complete the proof, it is enough to show that $R'$ satisfies
\begin{equation}\label{eq-14}
    X=X_0+B+M',\quad R'\as
\end{equation}
with $M'$ a local $R'$-martingale. Indeed, the condition (U)
tells
us that $R'=R,$ which implies that $R^{\prime\tau}=R^\tau.$ But
$R^{\prime\tau}=Q,$ hence
$Q=R^\tau.$

Let us show \eqref{eq-14}. Let $\xi\in\Rd$ and a stopping time
$\sigma$ be given. We denote $\seq{\tau}{n}$ a localizing
sequence
of $M=M^R$ and $B=B^R.$ Then,
\begin{eqnarray*}
  && E_{R'}[\xi\cdot(\xtn_\sigma-\xtn_0)]\\
&=&
E_{R'}[\1_{\{\tau\le\sigma\}}\xi\cdot(\xtn_\sigma-\xtn_\tau)]
  +E_{Q}[\xi\cdot(\xtn_{\sigma}-\xtn_0)] \\
&=& \IO
E_R[\1_{\{\tau\le\sigma\}}\xi\cdot(\xtn_\sigma-\xtn_\tau)\mid
\XXX0\tau=\eta]\, Q(d\eta)
  + E_{Q}[\xi\cdot(\xtn_{\sigma}-\xtn_0)] \\
&=& \IO
E_R[\1_{\{\tau\le\sigma\}}\xi\cdot(B_{\sigma}^{\tau_n}-B^{\tau_n}_\tau)\mid
\XXX0\tau=\eta]\, Q(d\eta)
  + E_{Q}[\xi\cdot(B_{\sigma}^{\tau_n}-B^{\tau_n}_0)]\\
  &=& E_{R'}[\xi\cdot(B_{\sigma}^{\tau_n}-B^{\tau_n}_0)]
\end{eqnarray*}
This means that \eqref{eq-14} is satisfied (with the localizing
sequence $\seq\tau n$) and completes the proof of the lemma.
\end{proof}

For all $k\ge1,$ we consider the stopping time
$$
\tau_k=\inf\left\{t\in[0,1]; \II0t \beta_s\cdot A(ds)\beta_s\ge
k\right\}\in [0,1]\cup \{\infty\}
$$
where $\beta$ is the process which is associated with $P$ in
Theorem \ref{res-01} and as a convention $\inf\emptyset=\infty.$ We are going to use this stopping time $R\as$ Since $\beta$ is only defined $P\as,$ we assume for the moment that $P$ and $R$ are equivalent measures: $P\sim R.$

\begin{lemma}\label{res-07}
Assume that $P\sim R$ and
suppose that $R$ satisfies the condition \emph{(U)}. Then, for
all
$k\ge1,$ on the  stochastic interval $\llbracket
0,\tau_k\wedge1\rrbracket$ we have, $R$-almost everywhere
\begin{equation}\label{eq-15}
    \1_{\llbracket 0,\tau_k\wedge1\rrbracket}\frac{dP}{dR}
= \1_{\llbracket
0,\tau_k\wedge1\rrbracket}\frac{dP_0}{dR_0}(X_0)
\exp\left(\II0{\tau_k\wedge1}\beta_t\cdot
dM^R_t-\frac12\II0{\tau_k\wedge1}\beta_t\cdot A(dt)\beta_t
    \right).
\end{equation}
\end{lemma}

\begin{proof}
By conditionning with respect to $X_0,$ we see that we can assume without loss of generality, that $R_0:=(X_0)\pf R=(X_0)\pf P=:P_0,$ i.e.\ $\frac{dP_0}{dR_0}(X_0)=1.$
Let  $k\ge1.$ Denote  $R^k=R^{\tau_k},$ $P^k=P^{\tau_k}.$
Applying Lemma \ref{res-05} with $\gamma=-\beta$ and  remarking
that $\widehat{B}_{-\beta}=-\widehat{B}_\beta,$ we see that
$$
 Q^k:=\mathcal{E}(-\beta\cdot M^P)_{\tau_k\wedge1}P^k
\in \MP(\1_{\llbracket 0,\tau_k\rrbracket}[(B^R+\widehat{B}_\beta)+\widehat{B}_{-\beta}],\1_{\llbracket 0,\tau_k\rrbracket}A))
=\MP(\1_{\llbracket 0,\tau_k\rrbracket}B^R,\1_{\llbracket 0,\tau_k\rrbracket}A). 
$$
But, it is known with
Lemma \ref{res-06} that $R^k$ satisfies the condition (U).
Therefore, 
\begin{equation}\label{eq-16}
    Q^k=R^k.
\end{equation}
Applying twice  Lemma \ref{res-05}, we observe on the one hand
that
\begin{equation}\label{eq-17}
    \widetilde{P}^k:=\mathcal{E}(\beta\cdot M^R)_{\tau_k\wedge1}R^k\in\MP(\1_{\llbracket 0,\tau_k\rrbracket}(B^R+\widehat{B}_\beta),\1_{\llbracket 0,\tau_k\rrbracket}A),
\end{equation}
and on the other hand that
$$
\widetilde{Q}^k:=\mathcal{E}(-\beta\cdot M^P)_{\tau_k\wedge1}\widetilde{P}^k
\in\MP(\1_{\llbracket 0,\tau_k\rrbracket}[(B^R+\widehat{B}_\beta)-\widehat{B}_\beta],\1_{\llbracket 0,\tau_k\rrbracket}A)=\MP(\1_{\llbracket 0,\tau_k\rrbracket}B^R,\1_{\llbracket 0,\tau_k\rrbracket}A).
$$
As for the proof of
\eqref{eq-16}, the condition (U) which is
satisfied by $R^k$ leads us to
    $
   \widetilde{ Q}^k=R^k.
    $
Therefore, we see with \eqref{eq-16} that $Q^k=\widetilde{Q}^k,$
i.e.\
$\mathcal{E}(-\beta\cdot
M^P)_{\tau_k\wedge1}P^k=\mathcal{E}(-\beta\cdot
M^P)_{\tau_k\wedge1}\widetilde{P}^k.$ And since
$\mathcal{E}(-\beta\cdot M^P)_{\tau_k\wedge1}>0,$  we obtain
$P^k=\widetilde{P}^k$
which is \eqref{eq-15}.
\end{proof}

We are  ready to complete the proof of Theorem \ref{res-04}.

\begin{proof}[Proof of Theorem \ref{res-04}. Derivation of $\dPR$]
Provided that $R$ satisfies the condition (U), when $P\sim R$ we obtain the
announced formula 
\begin{equation}\label{eq-50}
 \frac{dP}{dR}
=\frac{dP_0}{dR_0}(X_0)
\exp\left(\II0{1}\beta_t\cdot dM^R_t-\frac12\II0{1}\beta_t\cdot
A(dt)\beta_t \right),
\end{equation}
letting $k$ tend to infinity
in \eqref{eq-15}, remarking that $\tau:=\Lim k \tau_k=\inf\{t\in\ii;\II0t\beta_s\cdot A(ds)\beta_s=\infty\}$ and that \eqref{eq-32} implies
\begin{equation}\label{eq-21}
    \tau=\infty,\ P\as
\end{equation}  and, since $P\sim R,$ we also have  $\tau=\infty,$ $R\as$
Indeed, since $\tau(\omega)=\infty,$ there is some $k_o\ge1$ such that $\tau_{k_o}(\omega)=\infty$ and applying  Lemma \ref{res-07} with $k=k_o:$  $\dPR(\omega)=\frac{dP_0}{dR_0}(\omega_0)\exp\left(\II0{1}\beta_t\cdot dM^R_t-\frac12\II0{1}\beta_t\cdot A(dt)\beta_t   \right)(\omega)>0.$

Now, we consider the general case when $P$ might not be equivalent to $R.$
The main idea is to approximate $P$ by a sequence $\seq Pn$ such that $P_n\sim R$ for all $n\ge1,$ and to rely on our previous intermediate results. We consider
$$
	P_n:=\Big(1-\frac{1}{n}\Big)P+\frac{1}{n}R,\quad n\ge1.
$$
Clearly, $P_n\sim R$ and by convexity $H(P_n|R)\le (1-\frac{1}{n})H(P|R)+\frac{1}{n}H(R|R)\le H(P|R)<\infty.$ More precisely, the function $x\in\ii\mapsto H(xP+(1-x)R|R)\in[0,\infty]$ is a finitely valued convex continuous and increasing. It follows that
	$ 
\Lim n H(P_n|R)=H(P|R).
	$ 
\\
It is clear that $\Lim n P_n=P$ in total variation norm.
Let us prove that the stronger convergence
\begin{equation}\label{eq-48}
\Lim n H(P|P_n)=0
\end{equation}
also holds. It is easy to check that $\1_{\{\dPR\ge1\}} dP/dP_n$ and $\1_{\{\dPR\le1\}} dP/dP_n$ are respectively decreasing and increasing sequences of functions. It follows by monotone convergence that
\begin{multline*}
\Lim n H(P|P_n)=\Lim n\int\log(dP/dP_n)\,dP\\
	=\Lim n \int_{\{\dPR\ge1\}} \log(dP/dP_n)\,dP 
		+\Lim n\int_{\{\dPR<1\}} \log(dP/dP_n)\,dP= 0.
\end{multline*}
By Theorem \ref{res-01}, there exist two vector fields $\beta^n$ and $\beta$ which are respectively defined $R\as$ and $P\as$ such that $E_{P_n}\II01 \beta^n_t\cdot A(dt)\beta^n_t<\infty$, $E_P\II01 \beta_t\cdot A(dt)\beta_t<\infty$ and 
$$
 dX_t = dB^R_t+A(dt)\beta^n_t+dM_t^{P_n},\ R\as;
\qquad
dX_t = dB^R_t+A(dt)\beta_t+dM_t^{P}, P\as
$$
where $M^{P_n}$ and $M^P$ are respectively a local $P_n$-martingale and a local $P$-martingale. Therefore, 
\begin{equation}\label{eq-49}
dM^{P_n}_t=dM^P_t+A(dt)(\beta_t-\beta^n_t),\quad P\as
\end{equation}
Extending $\beta$ arbitrarily by $\beta=0$ on the $P$-null set where it is unspecified, we know that $$
\exp \left(\II0t (\beta_s-\beta_s^n)\cdot dM^{P_n}_s-\frac{1}{2}\II0t (\beta_s-\beta_s^n)\cdot A(ds)(\beta_s-\beta_s^n)\right)$$
is a  $P^n$-supermartingale. It follows with Proposition \ref{res-02}, \eqref{eq-49} and a standard monotone convergence argument that 
\begin{multline*}
H(P|P_n)
\ge E_P\left(\II01 (\beta_s-\beta_s^n)\cdot dM^{P_n}_s-\frac{1}{2}\II01 (\beta_s-\beta_s^n)\cdot A(ds)(\beta_s-\beta_s^n)\right)\\
=\frac{1}{2} E_P\II01 (\beta_s-\beta_s^n)\cdot A(ds)(\beta_s-\beta_s^n).
\end{multline*}
With \eqref{eq-48}, this shows the key estimate 
\begin{equation}\label{eq-51}
\Lim n E_P\II01 (\beta_s-\beta_s^n)\cdot A(ds)(\beta_s-\beta_s^n)=0.
\end{equation}
Since $H(P_n|R)<\infty$ and $P_n\sim R,$ under the condition (U) we can invoke \eqref{eq-50} to write
\begin{equation*}
 \frac{dP_n}{dR}
=\frac{dP_{n,0}}{dR_0}(X_0)
\exp\left(\II0{1}\beta^n_t\cdot dM^R_t-\frac12\II0{1}\beta^n_t\cdot
A(dt)\beta^n_t \right).
\end{equation*}
As $\Lim nP_n=P$ in total variation norm, up to the extraction of a $R\as$-convergent subsequence we have $\Lim n dP_n/dR=dP/dR$ and $\Lim n dP_{n,0}/{dR}=dP_0/dR_0.$ On the other hand, \eqref{eq-51} implies that $P\as,$ $\Lim n \frac12\II0{1}\beta^n_t\cdot
A(dt)\beta^n_t =\frac12\II0{1}\beta_t\cdot
A(dt)\beta_t.$   It follows that 
\begin{equation*}
 \frac{dP}{dR}
=\1_{\{\dPR>0\}}\frac{dP_0}{dR_0}(X_0)
\exp\left(\II0{1}\beta_t\cdot dM^R_t-\frac12\II0{1}\beta_t\cdot
A(dt)\beta_t \right).
\end{equation*}
where \eqref{eq-51} also implies that the limit of the stochastic integrals
\begin{equation*}
\Lim n \II0{1}\beta^n_t\cdot dM^R_t=\II0{1}\beta_t\cdot dM^R_t,\ P\as
\end{equation*}
exists $P\as$
\end{proof} 

It remains to compute $H(P|R).$

\begin{proof}[End of the proof of Theorem \ref{res-04}. Computation of $H(P|R)$] 
Let us first compute $H(P|R)$ when $R$ satisfies (U). Remark that in
the
proof of Lemma \ref{res-07}, for all $k\ge1$ the local
$\widetilde{P}^k$-martingale $N^k=M^R-\widehat B$ which is behind
\eqref{eq-17}
is a genuine martingale. It is a consequence of the first
statement of Lemma \ref{res-05}. As $\widetilde{P}^k=P^k,$ $N^k$
is a genuine
$P^k$-martingale. This still holds when $P\sim R$ fails. Indeed, this hypothesis has only been invoked to insure that $\tau_k$ is well-defined $R\as$ But in the present situation, $\tau_k$ only needs to be defined $P\as$ With \eqref{eq-15}, we have
\begin{eqnarray*}
 H(P^k|R^k)
    &=&E_{P^k}\log\frac{dP^k}{dR^k} \\
&\overset{\eqref{eq-15}}=&
E_P\left(\log\frac{dP_0}{dR_0}(X_0)\right)
+E_{P^k}\left(\II0{1}\beta_t\cdot
dM^R_t-\frac12\II0{1}\beta_t\cdot A(dt)\beta_t
    \right) \\
  &\overset{\eqref{eq-17}}=& H(P_0|R_0)
+E_{P^k}\left(\II0{1}\beta_t\cdot
(dN^k_t+d\widehat{B}_t)-\frac12\II0{1}\beta_t\cdot A(dt)\beta_t
    \right) \\
    &\overset{\eqref{eq-18}}=& H(P_0|R_0)
  +\frac12 E_{P^k}\left(\II0{1}\beta_t\cdot A(dt)\beta_t
    \right)+E_{P^k}\left(\II0{1}\beta_t\cdot dN^k_t\right) \\
    &=& H(P_0|R_0)
+\frac12 E_{P}\left(\II0{\tau^k\wedge1}\beta_t\cdot
A(dt)\beta_t\right)\\
\end{eqnarray*}
where the last equality comes from the $P^k$-martingale property of $N^k$. It remains to let $k$ tend to infinity to see that
$$
    H(P|R)=H(P_0|R_0)
  +\frac12 E_{P}\left(\II0{1}\beta_t\cdot A(dt)\beta_t\right).
$$
Indeed, because of \eqref{eq-21} and since the sequence
$\seq\tau
k$ is increasing, we obtain by monotone convergence that
$$
    \Lim k
E_{P}\left(\II0{\tau^k\wedge1}\beta_t\cdot A(dt)\beta_t\right)
=\frac12 E_{P}\left(\II0{1}\beta_t\cdot A(dt)\beta_t\right).
$$
As regards the left hand side of the equality, with Proposition
\ref{res-02}-(1) and \eqref{eq-21}, we see that
\begin{eqnarray*}
  H(P|R)
  &=&\sup\{E_P u(X)-\log E_R e^{u(X)}; u\in L^\infty(P)\}\\
&=& \sup_k\sup\{E_P u(X^{\tau^k})-\log E_R e^{u(X^{\tau_k})};
u\in L^\infty(P)\} \\
  &=& \Lim k H(P^k|R^k).
\end{eqnarray*}

It remains to check that, without the condition (U), we have
\begin{equation}\label{eq-23}
    H(P|R)\ge H(P_0|R_0)+\frac12 E_{P}\left(\II0{1}\beta_t\cdot
A(dt)\beta_t\right).
\end{equation}
Let us extend $\beta$ by $\beta=0$ on the $P$-null set where it is unspecified and define
\begin{equation*}
  \tilde{u}(X)
:=\log\frac{dP_0}{dR_0}(X_0)+\II0{\tau^k\wedge1}\beta_t\cdot
dM^R_t
    -\frac12 \II0{\tau^k\wedge1}\beta_t\cdot A(dt)\beta_t .
\end{equation*}
Choosing $\tilde u(X)$ at inequality  $\overset{(\textrm{i})}\ge$ below, thanks to an already used supermartingale argument, we
obtain the inequality $\overset{(\textrm{ii})}\ge$  below and
\begin{eqnarray*}
   H(P^k|R^k)
&\overset{\eqref{eq-11}}=&\sup\left\{\int u\,dP^k-\log\int
e^u\,dR^k; u: \int e^u\,dR^k<\infty\right\} \\
&\overset{(\textrm{i})}\ge& \int \tilde{u}\,dP^k-\log\int
e^{\tilde{u}}\,dR^k \\
  &\overset{(\textrm{ii})}\ge&  \int \tilde{u}\,dP^k\\
&\overset{(\textrm{iii})}=&
H(P_0|R_0)+E_{P^k}\left(\II0{\tau^k\wedge1}\beta_t\cdot
d\widehat{B}_t
    -\frac12 \II0{\tau^k\wedge1}\beta_t\cdot
    A(dt)\beta_t\right)\\
&\overset{\eqref{eq-18}}=& H(P_0|R_0)+\frac12
E_{P^k}\II0{\tau^k\wedge1}\beta_t\cdot
  A(dt)\beta_t.
\end{eqnarray*}
Equality (iii) is a
consequence of
$$
\tilde{u}(X)=
\log\frac{dP_0}{dR_0}(X_0)+\II0{\tau^k\wedge1}\beta_t\cdot
  (dM^P_t+d\widehat{B}_t)
    -\frac12 \II0{\tau^k\wedge1}\beta_t\cdot A(dt)\beta_t, \quad P^k\as
$$
which comes from Theorem \ref{res-01}. It remains to let $k$
tend
to infinity, to obtain as above with \eqref{eq-21} that \eqref{eq-23} holds true.
This completes the proof of the theorem.
\end{proof}

\section{Proofs of Theorems \ref{res-01b} and \ref{res-04b}}
We begin recalling It\^o's formula. Let $P$ be the law of a
semimartingale
\begin{equation*}
    dX_t=b_t\rho(dt)+dM^P_t
\end{equation*}
with $M^P$ a local $P$-martingale such that $M^P=q\odot\mutk $, $P\as$ That is $P\in\LK(\Kb)$ for some
L\'evy kernel $\Kb.$  For any $f$ in $\mathcal{C}^2(\Rd)$ which
satisfies:
  \par\medskip
\noindent($*$)\quad \emph{When localizing with an increasing
sequence $\seq{\tau}{k}$ of stopping times tending $P$-almost surely to infinity, for each $k\ge1$ the truncated process $\1_{\{|q|>1\}}\1_{\{t\le \tau_k\}}[f(\lg
{X}t+q)-f(\lg{X}t)]$ is a
$\mathcal{H}_1(P,\Kb)$ integrand},
    \par\medskip\noindent
It\^o's formula is
\begin{eqnarray}\label{eq-27}
    df(X_t)
&=&\Big[\int_{\Rd_*}[f(\lg Xt+q)-f(\lg Xt)-\nabla
f(\lg Xt)\cdot q
    ]\,\KK_t(dq)\Big]\,\rho(dt)\nonumber\\
    &&\quad+\nabla f(\lg Xt)\cdot b_t\,\rho(dt) +dM_t,\quad P\as
\end{eqnarray}
where $M$ is a local $P$-martingale. This identity would fail if $\rho$ was not assumed to be atomless.

\subsection*{Proof of Theorem \ref{res-01b}}

Based on It\^o's formula, we start computing a large family of exponential local martingales. Recall that we denote
$$
	a\mapsto \theta(a):=e^a-a-1=\sum_{n\ge2}a^n/n!,
\quad a\in\RR.
$$

\begin{lemma}[Exponential martingale]\label{res-03c}
Let $h:\Omega\times\ii\times\Rd_*\to\RR$ be a real valued
predictable process which satisfies 
\begin{equation}\label{eq-28}
    E_R \int_{\ii\times\RR_*}\theta[h_t(q)]\,\Lb(dtdq)<\infty.
\end{equation}
Then, $h$ and $e^{h}-1$ belong to $\mathcal{H}_{1,2}(R,\Lb )$. In particular, $h\odot\mut$ is a $R$-martingale.
\\
Moreover,
\begin{equation*}
Z^h_t:=\exp\Big(h\odot\mut_t-\int_{(0,t]\times\Rd_*}\theta[h_s(q)]\,
    \Lb(dsdq)\Big),\quad t\in\ii
\end{equation*}
is  a local $R$-martingale and a positive $R$-supermartingale which satisfies
\begin{equation*}
dZ^h_t= Z^h_{t^-}\,[(e^{h(q)}-1)\odot d\mut _t].
\end{equation*}
\end{lemma}

\proof
The function $\theta$ is nonnegative, quadratic near zero, linear near $-\infty$ and it grows exponentially fast near $+\infty.$ Therefore,  \eqref{eq-28} implies that $h$ and $e^{h}-1$ belong to $\mathcal{H}_{1,2}(R,\Lb )$. In particular, $M^h:=h\odot\mut$ is a $R$-martingale.
\\
Let us denote $Y_t=M^h_t-\int_{(0,t]}\beta_s\,\rho(ds)$ where
$\beta_t=\int_{\Rd_*}\theta[h_t(q)]\,
\LL _t(dq).$ Remark that \eqref{eq-28} implies that these integrals are almost everywhere well-defined. Applying \eqref{eq-27} with $f(y)=e^y$ and
$dY_t=-\beta_t\,\rho(dt)+dM^h_t$,
we obtain
$$
    de^{Y_t}=e^{\lg
Yt}\Big[-\beta_t+\int_{\Rd_*}\theta[h_t(q)]\,\LL_t(dq)\Big]\,\rho(dt)+dM_t
    =dM_t
$$
where $M$ is a local martingale. We are allowed to do this
because
($*$) is satisfied. Indeed, with $f(y)=e^y,$ 
	$
	f(\lg Yt+h_t(q))-f(\lg Yt)- f'(\lg Yt) h_t(q)
	=e^{\lg Yt}\theta[h_t(q)]
	$ 
and if
$Y^\sigma_t:=Y_{t\wedge \sigma}$ is stopped at
$\sigma:=\inf\{t\in\ii;Y_t\not\in C\}\in\ii\cup\{\infty\}$ for some compact subset $C$ with
the
convention $\inf\emptyset=\infty,$ we see with \eqref{eq-28} and
the fact that any path
in $\OO$ is bounded, that $\exp(\lg
{Y^\sigma}t)\theta[h_t(q)]$ is in $\mathcal{H}_1(R,\Lb).$  Now, choosing
the compact set $C$ to be the ball of radius $k$ and letting $k$
tend to infinity, we
obtain an increasing sequence of stopping
times $\seq \sigma k$ which tends almost surely to infinity. This
proves that $Z^h:=e^Y$ is a local martingale.
\\
We see that 
	$
dM_t=e^{\lg Yt}\,d[(e^{h(q)-1})\odot\mut _t],
$
keeping track of the martingale terms in the above differential
formula:
\begin{eqnarray*}
  de^{Y_t}
  &=& e^{\lg Yt}\big[\theta(\Delta Y_t)+dY_t\big] \\
&=& e^{\lg Yt}\big[\theta[h_t(q)]\odot
d\mut _t+\Big(\int_{\Rd_*}\theta[h_t(q)]\,\LL_t(dq)\Big)\rho(dt)
-\beta_t\,\rho(dt)+h(q)\odot d\mut _t\big]
\\&=& e^{\lg
Yt}\big[\theta[h_t(q)]\odot
d\mut _t+h_t(q)\odot
  d\mut _t\big]\\
&=& e^{\lg Yt}\big[(e^{h_t(q)}-1)\odot
d\mut _t\big].
\end{eqnarray*}
By Fatou's lemma, any nonnegative local martingale is also a supermartingale.
\endproof

\begin{proof}[Proof of Theorem \ref{res-01b}]
It follows the same line as the proof of Theorem \ref{res-01}. 
By Lemma \ref{res-03c},
    $
0<E_RZ^h_1\le1$ for all $h$ satisfying the assumption \eqref{eq-28}. By \eqref{eq-11}, for any probability measure $P$
such that $H(P|R)<\infty,$ we have
\begin{equation*}
E_P\left(h\odot\mut_1-\IiR \theta(h)\,d\Lb  \right)\le
H(P|R).
\end{equation*}
As in the proof of Theorem \ref{res-01}, see that
\begin{equation*}
|E_P(h\odot\mut_1)|\le (H(P|R)+1) \|h\|_{\theta},\quad \forall
h\end{equation*}
where 
\begin{equation}\label{eq-30}
\|h\|_{\theta}:=\inf\left\{a>0;E_P\IiR \theta(h/a)\,d\Lb  \le1\right\}\in[0,\infty]
\end{equation}
is the Luxemburg norm of the  Orlicz space 
\begin{equation*}
 L_\theta:=\Big\{h:\iRO\to\RR;\textrm{measurable }
\textrm{s.t.}\ E_P \IiR
\theta(b_o|h|)\,d\Lb <\infty, \textrm{for some }
b_o>0\Big\}.
\end{equation*} 
It differs from the corresponding \emph{small}  Orlicz space
\begin{equation*}
S _\theta:=\Big\{h:\iRO\to\RR;
\textrm{measurable s.t.}\ E_P \IiR
\theta(b|h|)\,d\Lb <\infty,\forall b\ge0\Big\}
\end{equation*}
because the function $\theta(|a|)$ grows exponentially fast. 

We introduce the space $\mathcal{B}$ of all the bounded
processes such that $ E_P \IiR|h| d\Lb <\infty,$ and its subspace $\mathcal{H}\subset \mathcal{B}$ which consists of the processes in $\mathcal{B}$ which are predictable.
We have $\mathcal{B}\subset
S _\theta$ and any $h$ in $\mathcal{H}$ satisfies \eqref{eq-28}, which is the
hypothesis of Lemma \ref{res-03c}. Hence, \eqref{eq-30} holds
for all $h\in \mathcal{H}$ and, as $H(P|R)<\infty,$ it tells us that the linear
mapping $h\mapsto E_P(h\odot\mut_1)$ is continuous on
$\mathcal{H}$ equipped with the norm $\|\cdot\|_\theta.$ Since
the convex conjugate of the Young function $\theta(|a|)$ is
$\theta^*(|b|),$ the dual space of
$(S _\theta,\|\cdot\|_\theta)$\footnote{This doesn't
hold with $ L_\theta$ instead of
$S _\theta$.} (see \cite{RaoRen}), is isomorphic to
\begin{equation*}
 L_{\theta^*}:=\Big\{k:\iRO\to\RR;
\textrm{measurable s.t.}\ E_P \IiR
\theta^*(|k|)\,d\Lb <\infty\Big\}.
\end{equation*} 	
Therefore, there exists some $k\in  L_{\theta^*}$ such that 
\begin{equation}\label{eq-31}
	E_P h\odot\mut_1=E_P\IiR kh\,d\Lb ,
\quad\forall h\in \mathcal{H}.
\end{equation}
Let us introduce the predictable projection $k^\mathrm{pr}$ of $k$ which is defined by $k^\mathrm{pr}_t:=E_P(k\mid \XXx0t),$ $t\in\ii.$ 
As the space $\mathcal{B}$ is dense in
$S _\theta$\footnote{In general, it is not dense in
$ L_\theta.$}, $\mathcal{H}$ is dense in the subspace of all predictable processes in $S _\theta$ and it follows that any $g$  and $k$ in $ L_{\theta^*}$ which both satisfy \eqref{eq-31}, share the same predictable projection: $g^\mathrm{pr}=k^\mathrm{pr}.$ Consequently, there is
a \emph{unique} predictable process $k$ in the space
\begin{equation*}
\mathcal{K}(P):=\Big\{k:\iRO\to\RR;
\textrm{predictable s.t.}\
E_P \IiR \theta^*(|k|)\,d\Lb <\infty \Big\}
\end{equation*}
which verifies \eqref{eq-31}.
\\
As  $\mathcal{H}$ is included in $\mathcal{H}_1(P,\Lb),$ we have for all $h\in \mathcal{H},$ $h\odot\mut-h\odot k\Lb=h\odot(\mul-\Lb -h\odot k\Lb  = h\odot(\mul-\ell\Lb )$ with $\ell:=k+1.$ Consequently, \eqref{eq-31} is equivalent to
\begin{equation}\label{eq-44}
E_P \big[h\odot(\mul-\ell\Lb )]=0,\quad\forall h\in\mathcal{H},
\end{equation}
which is the content of the theorem. It remains however to note that, being an
expectation of the positive measure $\mul$, $\ell \Lb $ is also a positive measure.
Therefore, $\ell$ is nonnegative. This completes the proof of the theorem.
\end{proof}

\subsection*{Proof of Corollary \ref{res-08}}
It is mainly a remark based on H\"older's inequality in Orlicz spaces.

\begin{proof}[Proof of Corollary \ref{res-08}]
We are under the exponential integrability  assumption \eqref{eq-35} and we denote $Z=\dPR.$ The finite entropy assumption \eqref{eq-02} is equivalent to $Z$ belongs to the Orlicz space $ L_{\theta^*}(R),$ i.e.\ $\|Z\|_{\theta^*,R}<\infty.$ H\"older's inequality in Orlicz spaces\footnote{It is an easy consequence of Fenchel's inequality: $|ab|\le \theta(|a|)+\theta^*(|b|),$ for all $a,b\in\RR.$} expressed with the Luxemburg norms (see \eqref{eq-30}) gives us for any nonnnegative random variable $U$: 
$E_P(U)=E_R(ZU)\le 2\|Z\|_{\theta^*,R}\|U\|_{\theta,R}.$ This quantity is finite if $\|U\|_{\theta,R}<\infty,$ and this is equivalent to $E_R(e^{a_o U})<\infty$ for some $a_o>0.$ As a consequence, \eqref{eq-35} implies that $E_P\IiR \1_{\{|q|\ge1\}} e^{b_o|q|}\,\Lb(dtdq)<\infty$ for some $b_o.$ But this is equivalent to: $\1_{\{|q|\ge1\}}|q|$ belongs to the Orlicz space $ L_\theta(P\otimes\Lb ).$ With \eqref{eq-29} we see that $(\ell-1)$ is in $ L_{\theta^*}(P\otimes\Lb )$ and by H\"older's inequality again, we obtain
$$
	E_P  \IiR \1_{\{|q|\ge1\}}|q||\ell(t,q)-1|\,\Lb(dtdq)<\infty.
$$
The small jump part: $
	E_P  \IiR \1_{\{|q|<1\}}|q||\ell(t,q)-1|\,\Lb(dtdq)<\infty,
$
is a direct consequence of H\"older's inequality in $ L_2.$ This proves \eqref{eq-36}.
\\
We write symbolically
$$
	\mut=\mu-\Lb 
	=\mu-\ell \Lb +(\ell-1)\Lb
	=\widehat\mu +(\ell-1)\Lb .
$$
Hence, $q\odot\mut=q\odot\widehat\mu+\int (\ell-1)q\,d\Lb$ provided that all these terms are well defined. But, we have assumed that $q\odot\mut$ is well-defined and we have just proved that $\int (\ell-1)q\,d\Lb$ is well-defined. Therefore, the remaining term is also well-defined and the proof is complete.
\end{proof}

\subsection*{Proof of Theorem \ref{res-04b}}

It is similar to the proof of Theorem \ref{res-04}. We begin with a tranfer result in the spirit of Lemma \ref{res-05}. Let $P$ be a probability measure on $\OO$ such that 
$$
	P\in\MP(B,\Kb)
$$
where $B$ is a continuous bounded variation adapted process and $\Kb$ is some L\'evy kernel $$\Kb(dtdq):=\rho(dt) \KK(t;dq).$$ Let $\lambda$ be a $[-\infty,\infty)$-valued predictable process on $\iR$ such that $\int_{\{\lambda\ge -1\}} \theta(\lambda)\,d\Kb<\infty$ and $\Kb(-\infty\le \lambda<-1)<\infty,$ $P\as$ We define for all $t\in\ii,$
$$
Z_t=\widetilde{\exp} \left(\lambda\odot \mutk _t-\int_{[0,t]\times\Rd_*} \theta(\lambda)\,d\Kb\right):=Z^+_tZ^-_t
$$
with
\begin{equation*}
\left\{
\begin{array}{rcl}
Z^+_t&=&\displaystyle{\exp\left(\lp\odot \mutk_t -\int_{(0,t]\times\Rd_*}\theta(\lp) d\Kb\right)}\\
Z^-_t&=&\displaystyle{\1_{\{t< \tau^\lambda\}} \exp\left(\sum_{0\le s\le t}\lm(s,\Delta X_s)-\int_{(0,t]\times\Rd_*}(e^{\lm}-1)\,  d\Kb\right)}
\end{array}
\right.
\end{equation*}
where 
$$
\lp=\1_{\{\lambda\ge -\alpha\}}\lambda,\qquad\lm=\1_{\{-\infty\le \lambda< -\alpha\}}\lambda
$$
 with $\alpha>0,$  $e^{-\infty}=0$ and $\tau ^{\lambda}=\inf_{}\{t\in\ii,\lambda(t,\Delta X_t)=-\infty\}.$  Remark that, although $Z^+$ and $Z^-$ both depend on the choice of $\alpha,$ their product $Z=Z^+Z^-$ doesn't depend on $\alpha>0.$
For all $j,k\ge1,$ we define
$$
	\sigma^k_j:=\inf\left\{t\in\ii; \int_{[0,t]\times\Rd_*} \theta(\lp)\, d\Kb\ge k
	\textrm{ or }\lambda(t,\Delta X_t)\not\in[-j, k]\right\}\in\ii\cup\{\infty\}
$$
and $P^k_j:={X^{\sigma^k_j}}_\#P.$

\begin{lemma}\label{res-05b}
Let  $P$ and $\lambda$ be as above. Then, for all $j,k\ge1,$
$Z^{\sigma^k_j}$ is a genuine $P$-martingale and the measure
$$
    Q^k_j:=Z^{\sigma^k_j}_1 P^k
$$
is a probability measure on $\Omega$ which satisfies 
\begin{equation*}
Q^k_j\in\MP\Big(B^{\sigma^k_j}+\widehat B^{\sigma^k_j},
\1_{\llbracket 0,\sigma^k_j\rrbracket}e^{\lambda}\Kb\Big)
\end{equation*}
where 
\begin{equation}\label{eq-34}
\widehat B_t=\int_{[0,t]\times\Rd_*}\1_{\{|q|\le1\}} (e^{\lambda}-1) q\,d\Kb,\quad t\in\ii.
\end{equation}
\end{lemma}

Note that $\widehat B_t$ might not be well defined in the general case. Only the stopped processes $\widehat B^{\sigma^k_j}$ are asserted to be meaningful.

\proof

Let us fix $j,k\ge1.$
We have  $Z^{\sigma^k_j}=\widetilde\exp(\lambda^k_j\odot \mutk  -\IiR \theta(\lambda^k_j)\,d\Kb)$  with $\lambda^k_j=\1_{\llbracket 0,\sigma^k_j\rrbracket}\lambda$ which is predictable since $\lambda$ is predictable and  $\1_{\llbracket 0,\sigma^k_j\rrbracket}$ is left continuous. We drop the subscripts and superscripts $j,k$ and write $\lambda=\lambda^k_j,$ $\lp=(\lambda^k_j)^+,$ $\lm=(\lambda^k_j)^-,$ $Z^{\sigma^k_j}=Z$ for the remainder of the proof. By the definition of $\sigma^k_j,$ we obtain with this simplified notation
\begin{equation}\label{eq-33}
	\IiR \theta(\lp)\,d\Kb \le k,\qquad -j\le\lambda\le k,\quad P^k_j\as
\end{equation}
Let us first prove that $Z$ is a $P^k_j$-martingale. Since it is a local martingale, it is enough to show that $$E_{P^k_j} Z^p_1<\infty,\quad \textrm{ for some } p>1.$$ 
\\
Choosing $\alpha=j$ in the definition of $(Z^{\sigma^k_j})^+$ and $(Z^{\sigma^k_j})^-,$ we see that $Z^{\sigma^k_j}=(Z^{\sigma^k_j})^+=Z^+=\mathcal{E}((e^{\lp}-1)\odot \mutk  ).$ 
For all $p\ge0,$ 
$$(Z^{+})^p=\exp\left(p\lp\odot \mutk  -p\IiR \theta(\lp)\,d\Kb \right)\le \exp(p\lp\odot \mutk  )$$ and
$$\mathcal{E}((e^{p\lp}-1)\odot \mutk  )=\exp\left(p\lp\odot \mutk  -\IiR \theta(p\lp)\,d\Kb \right)\ge 
e^{p\lp\odot \mutk  }/C(k,p)$$ for some finite deterministic constant $C(k,p)>0.$ To derive  $C(k,p),$ we must take account of \eqref{eq-33} and  rely upon the inequality $\theta(pa)\le c(k,p) \theta(a)$ which holds for all $a\in(-\infty,k]$ and some $0<c(k,p)<\infty.$ With this in hand, we obtain
$$
    (Z^+)^p\le e^{p\lp\odot \mutk  }\le C(k,p)\mathcal{E}((e^{p\lp}-1)\odot \mutk  ).
$$
We know with Lemma \ref{res-03c} that $\mathcal{E}((e^{p\lp}-1)\odot \mutk  )$
is a nonnegative local martingale. Therefore, it  is a
supermartingale. We deduce from this that
$E_{P^k_j}\mathcal{E}((e^{p\lp}-1)\odot \mutk  )\le 1$ and
$$
E_{P^k_j}(Z^+)^p\le C(k,p)E_{P^k_j}\mathcal{E}((e^{p\lp}-1)\odot \mutk  )\le C(k,p)<\infty.
$$
Choosing  $p>1,$ it follows that $\mathcal{E}((e^{\lp}-1)\odot \mutk  )$ is uniformly integrable. We conclude as in Lemma \ref{res-05}'s proof that $\mathcal{E}((e^{\lambda}-1)\odot \mutk  )$ is a genuine $P^k_j$-martingale.

Now, let us show that 
$$
Q^k_j\in\LK\Big(\1_{\llbracket 0,\sigma^k_j\rrbracket}e^{\lambda}\Kb\Big).
$$ 
Let $\tau$ be a finitely valued stopping time and $f$ a  measurable function on $\iR$ which will be specified later. We
denote $F_t=\sum_{0\le s\le t\wedge \tau}f(s,\Delta X_s)$ with the convention that $f(t,0)=0$ for all $t\in\ii.$ By  Lemma \ref{res-03c}, the martingale
$Z$  satisfies
    $dZ_t=\1_{\llbracket 0,\sigma^k_j\rrbracket}(t)\lg Zt[ (e^{\lambda}-1)\odot\mutk ]$. We have also 
    $dF_t=\1_{\llbracket 0,\tau\rrbracket}(t)f(t,\Delta X_t)$ and
$d[Z,F]_t=\1_{\llbracket 0,\sigma^k_j\wedge \tau\rrbracket}(t)\lg Zt (e^{\lambda(\Delta X_t)}-1)f(t,\Delta X_t),$ $P^k_j\as$
Consequently,
\begin{eqnarray*}
 && E_{Q^k_j}\sum_{0\le t\le  \tau}f(t,\Delta X_t)\\
  	&=& E_{P^k_j}(Z_{ \tau} F_{ \tau} -Z_0F_0) \\
	&=& E_{P^k_j}\II0{ \tau} (F_t\,dZ_t +Z_tdF_t+d[Z,F]_t) \\
	&=& E_{P^k_j}\left[\II0{ \tau} F_t\,dZ_t + 
	\sum_{0\le t\le  \tau}\lg Zt f(t,\Delta X_t)
			+\sum_{0\le t\le  \tau}\lg Zt (e^{\lambda(t,\Delta X_t)}-1)f(t,\Delta  X_t)\right] \\
 	&=& E_{P^k_j}\sum_{0\le t\le \tau}\lg Zt  e^{\lambda(t,\Delta X_t)}f(t,\Delta  X_t)\\
  &=& E_{P^k_j}\int_{[0,{ \tau}]\times\Rd_*} \lg Zt  f(t,q) e^{\lambda(t,q)}\,\Kb(dtdq)\\
	&=& E_{Q^k_j}\int_{[0,{ \tau}]\times\Rd_*}  f(t,q) e^{\lambda(t,q)}\,\Kb(dtdq).
\end{eqnarray*}
We are going to choose
$\tau$ such that the above terms are meaningful. For each $n\ge1,$ consider $\tau_n:=\inf_{}\{t\in\ii;\sum_{0\le s\le t\wedge \tau}|f(s,\Delta X_s)|\ge n\}$ and take $f$ in $L_1(P^k_j\otimes\Kb)$ to obtain $\Lim n \tau_n=\infty,$ $P^k_j\as$ and a fortiori $Q^k_j\as$ It remains to take $\tau=\sigma\wedge \tau_n$ with any stopping time $\sigma$ to see that the L\'evy kernel of $Q^k_j$ is $e^\lambda \Kb=e^{\lambda^k_j} \Kb.$

It remains to compute the drift term.
Let us denote $X^*_t:=\sum_{0\le s\le t}\1_{\{|\Delta X_s|>1\}}\Delta X_s$ the cumulated sum of large jumps of $X,$ and $X^{\triangle}:=X-X^*$ its complement.
Let $\tau$ be a finitely valued stopping time and
take $G_t=\xi\cdot X^\triangle_{t\wedge \tau}$ with $\xi\in\Rd.$ We have
    $dG_t=\1_{\llbracket 0,\tau\rrbracket}(t)\xi\cdot(dB_t+(\1_{\{|q|\le1\}} q)\odot d\mutk _t)$ and
$d[Z,G]_t=\1_{\llbracket 0,\sigma^k_j\wedge\tau\rrbracket}(t)\lg Zt (e^{\lambda(\Delta X_t)}-1)\1_{\{|\Delta X_t|\le1\}} \xi\cdot \Delta X_t,$ $P^k_j\as$
Therefore,
\begin{eqnarray*}
  &&E_{Q^k_j}[\xi\cdot(X^\triangle_\tau-X^\triangle_0)]\\
	&=& E_{P^k_j}[Z_\tau G_\tau -Z_0G_0] \\
	&=& E_{P^k_j}\left[\II0\tau (G_t\,dZ_t +Z_tdG_t+d[Z,G]_t)\right] \\
			&=& E_{P^k_j}\Big[\II0\tau G_t\,dZ_t + \II0\tau \lg Zt 		 \xi\cdot(dB_t+(\1_{\{|q|\le1\}}q)\odot d\mutk _t)\\
	&&\qquad\qquad\qquad\qquad	  +\sum_{0\le t\le \tau}\lg Zt \1_{\{|\Delta X_t|\le1\}} (e^{\lambda(t,\Delta X_t)}-1)\xi\cdot \Delta X_t\Big] \\
	&=& E_{P^k_j}\left[\II0\tau \lg Zt\xi\cdot dB_t+\sum_{0\le t\le \tau}\lg Zt   \1_{\{|\Delta X_t|\le1\}}(e^{\lambda(t,\Delta X_t)}-1)\xi\cdot \Delta X_t\right]\\
  &=&  E_{P^k_j}\left[\II0\tau \lg Zt\xi\cdot dB_t
	+\II0\tau \lg Zt  \Big\{\int_{\Rd_*} \1_{\{|q|\le1\}}(e^{\lambda(t,q)}-1)\xi\cdot q\,\KK_t(dq)\Big\}\,\rho(dt)\right]\\
&=&  E_{Q^k_j}  \II0\tau  \xi\cdot\left(dB_t+\Big\{\int_{\Rd_*} \1_{\{|q|\le1\}}(e^{\lambda(t,q)}-1) q\,\KK_t(dq)\Big\}\,\rho(dt)\right)
\end{eqnarray*}
where we take $\tau=\tau_n:=\inf_{}\{t\in\ii;|X_t|\ge n\}$ which tends to $\infty$ as $n$ tends to infinity. This shows that the drift term of $X$ under $Q^k_j$ is
$(B+\widehat B)^{\sigma_k}$ where $\widehat B$ is given at \eqref{eq-34} and the stopped process $\widehat B^{\sigma_k}$ is well-defined.
\endproof

As a first step, it is assumed that $P\sim R$ for the stopping times $\tau^k_j$, $\tau_j$ and $\tau^-$ to be  defined (below) $R\as$ and not only $P\as$

Following the proofs of Lemmas \ref{res-06} and \ref{res-07}, except for minor changes (but we skip the details), we arrive at analogous results:
\begin{enumerate}[(i)]
\item
If $R$ fulfills the uniqueness condition (U), then for any stopping time $\tau,$ $R^\tau$ also fulfills (U).
\item
If $P\sim R,$ then
for any $j,k\ge1,$ we have
\begin{equation*}
\1_{\llbracket 0,\tau^k_j\wedge1\rrbracket}\dPR
=\1_{\llbracket 0,\tau^k_j\wedge1\rrbracket}\frac{dP_0}{dR_0}(X_0)
\exp\left(
\big(\1_{(0,\tau^k_j\wedge1]}\log\ell\big)\odot \mut-\int_{(0,\tau^k_j\wedge1]\times\Rd_*}\theta(\log\ell)\,d\Lb\right)
\end{equation*}
where
$$
	\tau^k_j:=\inf\left\{t\in\ii; \int_{[0,t]\times\Rd_*} \1_{\{\ell>1/2\}}\theta(\log\ell)\, d\Lb\ge k
	\textrm{ or }\log\ell(t,\Delta X_t)\not\in[-j, k]\right\}\in\ii\cup\{\infty\}.
$$
\end{enumerate}
For the proof of (ii), we use Lemma \ref{res-03c} where $\lambda=\log\ell$ plays the same role as $\beta$ in Lemma \ref{res-07}, and we go backward with $-\lambda$ which corresponds to $\ell^{-1}$.

We fix $j,$ and let $k$ tend to infinity to obtain with \eqref{eq-29} that
$$
\Lim k \tau^k_j=\tau_j:=\inf\left\{t\in\ii; \ell(t,\Delta X_t)<e^{-j}\right\}\in\ii\cup\{\infty\}, \quad P\as
$$
and therefore $R\as$ also. More precisely, this increasing sequence is stationary after some time: there exists $K(\omega)<\infty$ such that $\tau^k_j(\omega)=\tau_j(\omega),$ for all $k\ge K(\omega).$ It follows that for all $j\ge1,$
\begin{multline}\label{eq-46}
\1_{\llbracket 0,\tau_j\wedge1\rrbracket}\dPR
=\1_{\llbracket 0,\tau_j\wedge1\rrbracket}\frac{dP_0}{dR_0}(X_0)
 \exp\left(
\big(\1_{(0,\tau_j\wedge1]}\log\ell\big)\odot \mut-\int_{(0,\tau_j\wedge1]\times\Rd_*}\theta(\log\ell)\,d\Lb\right).
\end{multline}

\begin{lemma}\label{res-09}
We do not assume that $P\sim R$ and we extend $\ell$ by $\ell=1$ on the $P$-negligible subset where it is unspecified. Defining $\tau^-:=\sup_{j\ge1} \tau_j,$ we have $P(\tau^-=\infty)=1.$
\end{lemma}

\begin{proof}
For all $j\ge1,$ we have $\tau^-\le1 \Rightarrow \sum_{t\le 1}\1_{\{\ell(t,\Delta X_t)\le e^{-j}\}}\ge1.$ Therefore,
\begin{multline*}
	P(\tau^-\le1)
\le P\left(\sum_{t\le 1}\1_{\{\ell(t,\Delta X_t)\le e^{-j}\}}\ge1\right)
\le E_P\sum_{t\le 1}\1_{\{\ell(t,\Delta X_t)\le e^{-j}\}}\\
\overset{\checkmark}{=} E_P\IiR  \1_{\{\ell\le e^{-j}\}}\,\ell d\Lb
\le e^{-j} E_P \Lb(\ell\le e^{-j})
\le e^{-j} E_P \Lb(\ell\le 1/2)
\end{multline*}
where we used \eqref{eq-44} at the marked equality.
The result will follow letting $j$ tend to infinity, provided that we show that $E_P \Lb(\ell\le 1/2)<\infty.$
\\
But, we know  with \eqref{eq-29} that $E_P\IiR \theta^*(|\ell-1|)\,d\Lb<\infty.$ Hence, $E_P\Lb(\ell\le 1/2)\le E_P\IiR \theta^*(|\ell-1|)\,d\Lb/\theta^*(1/2)<\infty$ and the proof is complete.
\end{proof}

\begin{lemma}\label{res-10}
Assume $P\sim R.$
	Let $R_j$ and $P_j$ be the laws of the stopped process $X^{\tau_j\wedge1}$ under $R$ and $P$ respectively. Then, under the condition \emph{(U)} we have for all $j\ge1$
$$
H(P_j|R_j)=H(P_0|R_0)+E_{P}\int_{(0,\tau_j\wedge1]\times\Rd_*}(\ell\log\ell-\ell-1)\,d\Lb.
$$
\end{lemma} 

\begin{proof}
We denote $R_j^k$ and $P_j^k$ the laws of the stopped process $X^{\tau_j^k\wedge1}$ under $R$ and $P$ respectively.  With the expression of $\dPR$ on $\llbracket 0,\tau_j^k\wedge1\rrbracket$ we see that
\begin{eqnarray*}
	&&H(P^k_j|R^k_j)\\
	&=& H(P_0|R_0)+E_{P^k_j}\left((\1_{( 0,\tau_j^k\wedge1]}\log\ell)\odot \mut-\int_{(0,\tau^k_j\wedge1]\times\Rd_*}\theta(\log\ell)\,d\Lb\right)\\
	&=&  H(P_0|R_0)+E_{P^k_j}\left((\1_{( 0,\tau_j^k\wedge1]}\log\ell)\odot \muh+\int_{(0,\tau^k_j\wedge1]\times\Rd_*}[(\ell-1)-\theta(\log\ell)]\,d\Lb\right)\\
	&=& H(P_0|R_0)+E_{P_j}\int_{(0,\tau^k_j\wedge1]\times\Rd_*}(\ell\log\ell-\ell-1)\,d\Lb
\end{eqnarray*}
where we invoke Lemma \ref{res-05b} at the last equality. We complete the proof letting $k$ tend to infinity.
\end{proof}

\begin{proof}[Conclusion of the proof of Theorem \ref{res-04b}]

When $P\sim R,$ by Lemma \ref{res-09}, $P$-almost surely there exists $j_o$ large enough such that for all $j\ge j_o,$ $\tau_j=\infty$ and \eqref{eq-46} tells us that
\begin{equation*}
\dPR
=\frac{dP_0}{dR_0}(X_0)
\exp\left(
(\log\ell)\odot \mut-\IiR\theta(\log\ell)\,d\Lb\right)
\end{equation*}
and also that the product appearing in $Z^-$ contains $P$-almost surely a finite number of terms which are all positive. Note that we do not use any limit result for stochastic or standard integrals; it is an immediate $\omega$-by-$\omega$ result with a stationary sequence. This is the desired expression for $\dPR$ when $P\sim R.$

Let us extend this result to the case when $P$ might not be equivalent to $R.$ We proceed exactly as in Theorem \ref{res-04}'s proof and start from \eqref{eq-48}: $\Lim n H(P|P_n)=0$ where $P_n:=(1-1/n)P+R/n,$ $n\ge1.$ Let us write $\lambda=\log\ell$ and $\lambda^n=\log\ell^n$ which are well-defined $P\as$ Thanks to Theorem \ref{res-01b}, we see that 
\begin{eqnarray*}
H(P|P_n)
&\ge& E_P \left((\lambda-\lambda^n)\odot \widetilde{\mu}^{\ell ^nL}
	-\IiR \theta(\lambda-\lambda^n)\,\ell^n d\Lb\right)\\
&=& E_P \left((\lambda^n-\lambda)\odot \widetilde{\mu}^{\ell L}
	+\IiR [\ell/\ell^n\log(\ell/\ell^n)-\ell/\ell^n+1]\,\ell^nd\Lb\right)\\
&=& E_P \IiR [\ell^n/\ell -\log(\ell^n/\ell)-1]\,d\ell \Lb\\
&=& E_P\IiR \theta(\lambda^n-\lambda)\, d\ell\Lb
\end{eqnarray*}
which leads to the entropic estimate analogous to \eqref{eq-51}:
\begin{equation}\label{eq-52}
\Lim n E_P\IiR \theta(\lambda^n-\lambda)\, d\ell\Lb=0.
\end{equation}
Taking the difference between $\log(dP_n/dR)=\lambda^n\odot \mut-\IiR\theta(\lambda^n)\,d\Lb$ and the logarithm of the announced formula \eqref{eq-42} for $dP/dR$ on the set $\{\dPR>0\},$ we obtain
$$
(\lambda^n-\lambda)\odot \widetilde{\mu}^{\ell L}
-\IiR \theta(\lambda^n-\lambda)\, d\ell\Lb,\quad P\as
$$
and the desired convergence follows from \eqref{eq-52}. Note that $\theta(a)=a^2/2+o_{a\rightarrow 0}(a^2).$ This completes the proof of \eqref{eq-42}.
  
As in the proof of Theorem \ref{res-04}, we obtain the announced formula for $H(P|R)$ under the condition (U) with  Lemmas \ref{res-09} and \ref{res-10}, and the corresponding general inequality follows from choosing 
\begin{equation*}
  \tilde{u}(X)
:=\log\frac{dP_0}{dR_0}(X_0)+(\1_{(0,\tau^k_j\wedge1]}\log\ell)\odot \mut
    -\int_{(0,\tau^k_j\wedge1]\times\Rd_*}\theta^*(\log\ell)\,d\Lb 
\end{equation*}
in the variational representation formula \eqref{eq-11}, and then letting $k$ and $j$ tend to infinity.
\end{proof}

\appendix

\section{An exponential martingale with jumps}

Next proposition is about exponential martingale with jumps. We didn't use it during the proofs of this paper. But we give it here for having a more complete picture of the Girsanov theory.

In this result, integrands $h$ are considered which may attain the value $-\infty.$ This is because with $h= \log\ell$, $h=-\infty$ corresponds to $\ell=0.$

\begin{proposition}[Exponential martingale]\label{res-03b}
Let $h:\Omega\times\ii\times\Rd_*\to[-\infty,\infty)$ be an extended real valued
predictable process which may take the value $-\infty$ and satisfies 
\begin{eqnarray}
    &&E_R \int_{\ii\times\RR_*}\1_{\{h_t(q)\ge-1\}}\theta[h_t(q)]\,\Lb(dtdq)<\infty, \label{eq-28b}\\
&&E_R \int_{\ii\times\RR_*}\1_{\{h_t(q)<-1\}}\,\Lb(dtdq)<\infty.\label{eq-28c}
\end{eqnarray}
Let us introduce the stopping time 
\begin{equation*}
\tau^h:=\inf_{}\{t\in\ii;h(\Delta X_t)=-\infty\}\in\ii\cup \{\infty\}
\end{equation*}
and the convention $e^{-\infty}=0.$
\\
Then, $e^{h}-1$ is in $\mathcal{H}_{1,2}(R,\Lb)$ and
\begin{equation}\label{eq-37}
Z^h_t:=\1_{\{t< \tau^h\}}\widetilde\exp\Big(h\odot\mut_t-\int_{(0,t]\times\Rd_*}\theta[h_s(q)]\,
    \Lb(dsdq)\Big),\quad t\in\ii
\end{equation}
is  a local $R$-martingale and a nonnegative $R$-supermartingale which satisfies
\begin{equation}\label{eq-41}
dZ^h_t=\1_{\{t\le \tau^h\}} Z^h_{t^-}\,[(e^{h(q)}-1)\odot d\mut _t].
\end{equation}
\end{proposition}

The standard notation is $Z^h:=\mathcal{E}([e^h-1]\odot\mut),$ the stochastic exponential of $[e^h-1]\odot\mut.$ Some details are necessary to make precise the sense of the inner stochastic integral $h\odot\mut_t$ in the expression of $Z^h_t.$ We denote
\begin{eqnarray*}
h^+&:=&\1_{\{h\ge-1\}}h\in \RR\\
h^-&:=&\1_{\{h<-1\}}h\in [-\infty,0].
\end{eqnarray*}
Under the assumption \eqref{eq-28b}, $h^+\odot\mut$ is well defined as a stochastic integral. On the other hand, \eqref{eq-28c} implies that $h^-(t,\Delta X_t)$ has $R\as$ finitely many jumps. It follows that $\sum_{0\le s\le t}h^-(s,\Delta X_s)$ is meaningful for all $t< \tau^h$. But the integral $\int_{(0,t]\times\Rd_*}h^-_s(q)\,\Lb(dsdq)$ might not be defined under \eqref{eq-28c} and   $h^-\odot\mut_t=\sum_{0\le s\le t}h^-(s,\Delta X_s)-\int_{(0,t]\times\Rd_*}h^-_s(q)\,\Lb(dsdq)$ is meaningless in this case. Nevertheless, the full expression in the exponential $\zeta(h):= h\odot\mut-\int \theta(h)\,d\Lb$ is defined as follows. We put 
$	
\zeta(h^-)
:=\sum_{0\le s\le t}h^-(s,\Delta X_s)-\int_{(0,t]\times\Rd_*}[e^{h^-_s(q)}-1]\,
    \Lb(dsdq)
$	
which is well defined under \eqref{eq-28c} and is obtained by cancelling the terms $\int_{(0,t]\times\Rd_*}h^-_s(q)\,\Lb(dsdq)$. As $\theta(0)=0,$   we have $\zeta(h)=\zeta(h^++h^-)=\zeta(h^+)+\zeta(h^-)$ and for all $t\in\ii,$
\begin{equation}\label{eq-38}
\left\{
\begin{array}{rcl}
Z^h_t&=&Z^{h^+}_tZ^{h^-}_t \quad \textrm{with}\\
Z^{h^+}_t&:=&\exp\Big(h^+\odot\mut_t-\int_{(0,t]\times\Rd_*}\theta[h^+_s(q)]\,
    \Lb(dsdq)\Big),\\
Z^{h^-}_t&:=&\1_{\{t< \tau^h\}} \exp\left(\sum_{0\le s\le t}h^-(s,\Delta X_s)-\int_{(0,t]\times\Rd_*}[e^{h^-_s(q)}-1]\,
    \Lb(dsdq)\right).\\
\end{array}\right.
\end{equation}
This is what is meant by the concise expression \eqref{eq-37}.

\proof 
Now, we consider the general case where $h$ may attain the value $-\infty$ and \eqref{eq-28} is weakened by \eqref{eq-28b} and \eqref{eq-28c}. We use the decomposition \eqref{eq-38} and write $Z^+=Z^{h^+}$ and $Z^-=Z^{h^-}$ for short. Clearly, $Z^+$ and $Z^-$ do not jump at the same times and $d[Z^+,Z^-]=\Delta Z^+ \Delta Z^-=0.$ Hence,
\begin{equation}\label{eq-39}
dZ_t=\lg{Z^+}tdZ^-_t+\lg{Z^-}tdZ^+_t.
\end{equation}
The $h^+$-part enters the framework of Lemma \ref{res-03c} and we have 
\begin{equation}\label{eq-39a}
dZ^+_t=\lg{Z^+}t \Big([e^{h^+}-1]\odot\mut\Big).
\end{equation}
Let us look at the $h^-$-part. We need to compute $dZ^-_t.$ For all $t< \tau^h,$ put $$Y_t^-=\sum_{0\le s\le t}h^-(s,\Delta X_s)-\int_{(0,t]\times\Rd_*}[e^{h^-_s(q)}-1]\,
    \Lb(dsdq).$$
Then, with the convention that $h^-(t,0)=0,$ $dY^-_t=h^-(t,\Delta X_t)-\gamma_t\,\rho(dt)$ with $\gamma_t=\int_{\Rd_*}[e^{h^-_t(q)}-1]\,\LL _t(dq),$ $\Delta Y^-_t=h^-(t,\Delta X_t)$ and with It\^o's formula, we arrive at
\begin{multline*}
de^{Y^-_t}=e^{\lg {Y^-}t}\Big([e^{\Delta Y^-_t}-1]+dY^-_t-\Delta Y^-_t\Big)
	=e^{\lg {Y^-}t}\Big([e^{h^-(t,\Delta X_t)}-1]-\gamma_t\,\rho(dt)\Big)\\
	=e^{\lg {Y^-}t}\Big([e^{h^-}-1]\odot d\mut_t\Big).
\end{multline*}
It follows that 
\begin{equation}\label{eq-40}
dZ^-_t=\lg{Z^-}t\Big([e^{h^-}-1]\odot d\mut_t\Big),\quad t<\tau^h.
\end{equation}
At $t=\tau^h,$ by the definition \eqref{eq-38} of $Z^-,$ we have 
$$
dZ^-_{|t=\tau^h}=-\lg{Z^-}{(\tau^h)}=\lg{Z^-}{(\tau^h)}\times [e^{-\infty}-1]
$$
which is \eqref{eq-40} at $t=\tau^h$ with the convention $e^{-\infty}=0.$  This provides us with 
$$
dZ^-_t=\1_{\{t\le \tau^h\}} \lg{Z^-}t \Big([e^{h^-}-1]\odot\mut\Big).
$$
Together with \eqref{eq-39} and \eqref{eq-39a}, this proves \eqref{eq-41} which implies that $Z^h$ is a local $R$-martingale.
\\
By Fatou's lemma, any nonnegative local martingale is also a supermartingale.
\endproof


\end{document}